\documentclass[a4paper,10pt]{amsart}
\usepackage{amsmath,amsthm,amssymb,tikz,calc,graphicx}
\usepackage[latin1]{inputenc}
\usepackage[T1]{fontenc}
\usepackage[all]{xy}
\usepackage{mathrsfs,pifont}

\numberwithin{equation}{section}
\newtheorem{thm}[equation]{Theorem}
\newtheorem*{thm*}{Theorem}
\newtheorem{cor}[equation]{Corollary}
\newtheorem{lem}[equation]{Lemma}
\newtheorem{prop}[equation]{Proposition}

\newtheorem*{prop*}{Proposition}

\newtheorem*{cor*}{Corollary}
\newtheorem{warnrem}[equation]{Remark/Warning}
\theoremstyle{definition}

\newtheorem{rem}[equation]{Remark}
\newtheorem{defn}[equation]{Definition}

\makeatletter
\DeclareRobustCommand\smallop[2][1]{%
	\mathop{\vphantom{\oplus}\mathpalette\smallop@{{#1}{#2}}}\slimits@
}
\newcommand{\smallop@}[2]{\smallop@@#1#2}
\newcommand{\smallop@@}[3]{%
	\vcenter{%
		\sbox\z@{$#1\oplus$}%
		\hbox{\resizebox{\ifx#1\displaystyle#2\fi\dimexpr\ht\z@+\dp\z@}{!}{$\m@th#3$}}%
	}%
}
\makeatother

\makeatletter
\DeclareRobustCommand\bigop[2][1]{%
	\mathop{\vphantom{\bigoplus}\mathpalette\bigop@{{#1}{#2}}}\slimits@
}
\newcommand{\bigop@}[2]{\bigop@@#1#2}
\newcommand{\bigop@@}[3]{%
	\vcenter{%
		\sbox\z@{$#1\bigoplus$}%
		\hbox{\resizebox{\ifx#1\displaystyle#2\fi\dimexpr\ht\z@+\dp\z@}{!}{$\m@th#3$}}%
	}%
}
\makeatother

\newcommand{\bigtopdirsum}{%      
	\tikz[baseline]{\draw[thick] 
		(-.025,-1ex) -- ++(0,3.2ex) 
		(.025,-1ex) -- ++(0,3.2ex) 
		(0,-.025)++(-1.6ex,.6ex) -- ++(3.2ex,0) 
		(0,.025)++(-1.6ex,.6ex) -- ++(3.2ex,0) 
		(0,.6ex) circle (1.6ex) }%  
}

\newcommand{\bigtoplus}{\DOTSB\bigop[.86]{\bigtopdirsum}}

\def\R{\mathbb R}
\def\N{\mathbb N}
\def\Z{\mathbb Z}
\def\A{\mathbb A}
\def\Q{\mathbb Q}
\def\C{\mathbb C}

\def\P{\{P\}}
\def\ira{\stackrel{\sim}{\longrightarrow}}
\def\hra{\hookrightarrow}
\def\ra{\rightarrow}
\def\U{\mathcal U}
\def\g{\mathfrak g}

\def\n{\mathfrak n}
\def\a{\mathfrak a}

\def\l{\mathfrak l}

\def\p{\mathfrak p}

\def\frakh{\mathfrak h}
\def\<{\langle}
\def\>{\rangle}
\def\J{\mathcal J}
\def\Hom{{\rm Hom}}

\def\vv{{\sf v}}
\def\ira{\stackrel{\sim}{\longrightarrow}}
\def\ira{\stackrel{\sim}{\longrightarrow}}
\def\id{\mathrm{id}}

\providecommand{\absl}[1]{\left\lvert#1\right\rvert}
\providecommand{\norm}[1]{\left\lVert#1\right\rVert}

\newcommand{\calZ}{{\mathcal{Z}}}

\newcommand{\GL}{{\mathrm{GL}}}
\newcommand{\spacedcdot}{\,\cdot\,}
\providecommand{\norm}[1]{\left\lVert#1\right\rVert}
\providecommand{\ab}[1]{\left|#1\right|}
\newcommand{\calA}{{\mathcal A}}
\newcommand{\calP}{{\mathcal P}}
\newcommand{\calH}{{\mathcal H}}

\DeclareMathOperator{\Cl}{Cl} 
\DeclareMathOperator{\Lie}{Lie} 
\DeclareMathOperator{\vol}{vol}

\newcommand{\bigcprojtp}{\mathop{\overline{\bigotimes}_{\sf pr}}\limits}

\newcommand{\cprojtp}{\mathop{\overline{\otimes}_{\sf pr}}\limits}

\newcommand{\cindtp}{\mathbin{\overline{\otimes}_{\sf in}}}% completed ind. t. p.

\newcommand{\rtprod}{\mathop{%
    \mathchoice%
        {\sideset{}{'}\bigotimes}%
        {\bigotimes'}%
        {\bigotimes'}%
        {\bigotimes'}%
    }%
}

\title[]{On the notion of the parabolic and the cuspidal support of smooth-automorphic forms and smooth-automorphic representations}
\author{Harald Grobner \& Sonja \v Zunar}
\address{Harald Grobner: Fakult\"at f\"ur Mathematik\\ University of Vienna\\ Oskar-Morgenstern-Platz 1\\ A-1090 Wien\\Austria}
\email{harald.grobner@univie.ac.at}
\urladdr{https://homepage.univie.ac.at/harald.grobner}

\address{Sonja \v Zunar: Faculty of Science, Department of Mathematics\\ University of Zagreb\\ Bijeni\v cka cesta 30\\ 10000 Zagreb\\ Croatia}
\email{szunar@math.hr}
\urladdr{https://web.math.pmf.unizg.hr/~szunar/}

\thanks{The first-named author is supported by the Austrian Science Fund (FWF) START-prize Y966 and also by the FWF Stand-alone research project P32333. The second-named author is supported by the Croatian Science Foundation grant IP-2018-01-3628.}

\usepackage{xcolor}

\usepackage[normalem]{ulem}

\begin{document}
\maketitle

\begin{abstract}
In this paper we describe several new aspects of the foundations of the representation theory of the space of smooth-automorphic forms (i.e., not necessarily $K_\infty$-finite automorphic forms) for general connected reductive groups over number fields. Our role model for this space of smooth-automorphic forms is a ``smooth version'' of the space of automorphic forms, whose internal structure was the topic of Franke's famous paper \cite{franke}.  We prove that the important decomposition along the parabolic support, and the even finer -- and structurally more important -- decomposition along the cuspidal support of automorphic forms transfer in a topologized version to the larger setting of smooth-automorphic forms. In this way, we establish smooth-automorphic versions of the main results of \cite{franke_schwermer} and of \cite{moewal}, III.2.6.  
\end{abstract}

\setcounter{tocdepth}{1}
\tableofcontents
\section*{Introduction}
\subsection*{Context}

%Smooth-automorphic forms have proved indispensable in the formulation as well as in the recent proof of the global Gan--Gross--Prasad conjectures for unitary groups by Beuzart-Plessis--Chaudouard--Zydor \cite{BPCZ20}. Their paper being on unitary groups, in this paper we develop and describe the very foundations of the representation theory of smooth-automorphic forms of general connected reductive groups over number fields. 

It has been almost 30 years since J. Franke wrote his epochal paper \cite{franke} on the space of automorphic forms $\calA_\J([G])$ attached to a connected reductive group $G$ over a number field $F$ (and a fixed, but arbitrary ideal $\J$ of finite codimension). We investigate a ``smooth version'' of Franke's space of automorphic forms, to be denoted $\calA^\infty_\J([G])$ and to be called {\it smooth-automorphic forms}. This is to be understood as a certain topological completion of the original space $\calA_\J([G])$ -- by getting rid of the condition of $K_\infty$-finiteness of the elements of $\calA_\J([G])$ -- which provides us the benefit of now having the whole group $G(\A)$ act continuously by right translation and thus regaining the representation-theoretical symmetry between the archimedean and non-archimedean places of $F$. 

This idea of replacing classical $K_\infty$-finite automorphic forms by a ``smooth'' version of them (in the above sense) is not new, is well-known to the experts, and goes back to early ideas of Bernstein, Casselman, Wallach et al., which are in turn closely related to what one calls the Casselman-Wallach completion of Harish-Chandra modules\footnote{To quote from \cite{casselman_schwartz}, Rem.\ 3.6, which refers to the then newly developed theory of Casselman-Wallach completions: {\it This result suggests that one plausible, and perhaps useful, extension of the notion of automorphic form would be to include functions in $A_{umg}(\Gamma\backslash G)$ which are $Z(\g)$-finite but not necessarily $K$-finite.}}. 

More recently, smooth-automorphic forms gained significant importance by supplying the context in which the global Gan--Gross--Prasad conjectures were formulated, \cite{ggp}, \S 22, p.\ 80 and, even more importantly, by the indispensable role they played in the very recent proof of these conjectures for (endoscopic) unitary groups by Beuzart-Plessis--Chaudouard--Zydor, cf.\ \cite{BPCZ20}, in particular \S 2.7 therein. 

Still, only rather basic results are known about the representation theory of smooth-automorphic forms.

\subsection*{Results} In this paper we extend the theory of smooth-automorphic forms of a general connected reductive group $G$ over a number field by establishing two of the most important internal structural descriptions of the representation theory of Franke's space $\calA_\J([G])$ in the smooth-automorphic framework. More precisely, we show that the decomposition along the parabolic and the (more refined and hence structurally deeper) cuspidal support transfer in a topologized version to the limit-Fr\'echet (LF) space of smooth-automorphic forms $\calA^\infty_\J([G])$. Our main result can be summarized as follows:

\begin{thm*}
Let $\calA^\infty_\J([G])$ be the LF-space of smooth-automorphic forms $\varphi: G(F)A_G^\R\backslash G(\A)\ra\C$, which are annihilated by a power of $\J$. Then, we have the (bi-continuous) decompositions of $G(\A)$-representations as locally convex direct sums
$$\calA_\J^\infty([G])=\bigoplus_{\left\{P\right\}}\calA^\infty_{\J,\P}([G]) = \bigoplus_{\left\{P\right\}} \bigoplus_{\pi}\calA^\infty_{\J,\P,\pi}([G]) $$
along the parabolic and the cuspidal smooth-automorphic support. Here, $\calA^\infty_{\J,\P}([G])$ denotes the $G(\A)$-subrepresentation of smooth-automorphic forms, which are negligible along every parabolic subgroup, which is not in the associate class $\{P\}$ of the parabolic subgroup $P$; and $\calA^\infty_{\J,\P,\pi}([G])$ denotes the $G(\A)$-subrepresentation of smooth-automorphic forms, which are limits of derivatives of regularized values of Eisenstein series attached to the associate class of the cuspidal smooth-automorphic representation $\pi$ of $L_P(\A)$. 
\end{thm*}
As applied to the case $P=G$, the space $\calA^\infty_{\J,\{G\}}([G])=\calA^\infty_{cusp,\J}([G])$ of cuspidal smooth-automorphic forms decomposes as a countable locally convex direct sum over the LF-spaces $\calH^{\infty_\A}$ of globally smooth vectors in the irreducible Hilbert space summands $\calH$ of the unitary representation on the Hilbert-space $L^2_{cusp,\J}(G(F)A_G^\R\backslash G(\A))$. We refer to the body of the paper for the precise definitions of all terms and to Thm.\ \ref{thm:112}, \ref{thm:cusp} and \ref{thm:cuspsup} for a proof.

Aside our main result, one driving source of motivation to write this paper was to address and resolve several misunderstandings about the space of smooth-automorphic forms, that can be found in the (rather sparse) literature on the subject. A basic but at the same time crucial such misunderstanding concerns the question of how to topologize the space of smooth-automorphic forms $\calA^\infty_\J(G)$ (i.e., how to make it into a Hausdorff locally convex topological vector space (LCTVS)), so that the right regular action would naturally make it a representation of $G(\A)$. Some authors do not address this point at all, and unfortunately we had to learn that several other authors even use incorrect constructions, leading to incorrect proofs. This convinced us that one should also provide a concise, short treatment of the basics of the theory, i.e., provide a source where they are settled carefully. 

Just to give the reader two hints of why this does not amount to a simple exercise in carefulness: (1): Consider the space $C^\infty_{umg,d}(G(F)\backslash G(\A))$ of smooth, left-$G(F)$-invariant functions $f: G(\A)\ra\C$, which satisfy the growth condition $p_{d,X}(f):=\sup_{g\in G(\A)}\ab{(Xf)(g)}\,\norm g^{-d} <\infty$, for a fixed $d\in\N$ and varying $X\in\U(\g)$. We could name at least three references in the literature, in which it is claimed that the space $C^\infty_{umg,d}(G(F)\backslash G(\A))$, its subspace of $\mathcal Z(\g)$-finite functions $C^\infty_{umg,d}(G(F)\backslash G(\A))_{(\calZ(\g))}$ or its subspace of $\mathcal Z(\g)$-finite and right-$K_n$-invariant functions $C^\infty_{umg,d}(G(F)\backslash G(\A))^{K_n}_{(\calZ(\g))}$ becomes a Fr\'echet space, when equipped with the seminorms $p_{d,X}$. However, none of these spaces is complete in the respective topology, and so all the concepts to be found in the literature, which are based on such claims must fail from the very beginning. Instead, here we show that the space of smooth-automorphic forms $\calA^\infty_\J(G)$ can be given a natural LF-space topology, exploiting the non-trivial fact that the growth of all functions $\varphi\in\calA^\infty_\J(G)$ can be uniformly bounded by the same exponent of growth. This result, see our Prop.\ \ref{prop:same_d}, seems to be largely neglected in the relevant literature but is a necessary key to make all topological constructions work. (2): Another frequently occurring misconception concerns the role of the topology on irreducible smooth-automorphic representations $V\subset  \calA^\infty_\J(G)$: It is unfortunately often neglected that it does not follow automatically that $V$, when equipped with the subspace topology from $\calA^\infty_\J(G)$, inherits the structure of an LF-space. In fact, it was one of the striking early discoveries of Groethendieck that a closed subspace of an LF-space does not need to be an LF-space itself. However, the existence of an LF-space structure on $V$ is indispensable, in order to obtain a (topologized) version of the restricted tensor product theorem in the smooth-automorphic context, i.e., the very key to a local-global principle for irreducible smooth-automorphic representations. Here, we solve both problems, providing a certain global analogue of a famous result of Casselman and Wallach, see Prop.\ \ref{prop:smmf} and Thm.\ \ref{thm:TPthm}.

%In our first sections we address this subtle question by a novelty approach, providing the necessary groundwork for a smooth theory of automorphic representations in the sense of Franke. One advantage of our new approach is that it provides us with representations on {\it complete} LCTVS -- an instance of a topological aspect which seems to have been ignored in the available literature on the adelic setting -- while it makes all the usual standard results on smooth-automorphic representations true (admissibility, restricted tensor product decomposition).  TO BE CONTINUED
%Moreover, in passing we take the chance to correct several apparent mistakes in some of the existing references on the topic, when defining a topology on $\calA^\infty_\J(G)$ and its attached representations, mainly related to a (possibly too optimistic) view on the nature of LF-spaces (whose intricacies have been a prominent theme of Grothendieck's early reserach, cf.\ \cite{grothendieck_LF, grothendieck_book, grothendieck_topovs}). ... 

We express our hope that this paper will become a useful reference for anyone, who seeks a clear and concise treatment of the internal, representation-theoretical structure of the space of smooth-automorphic forms in complete generality.
\small
\subsubsection*{Acknowledgments:} We are grateful to Rapha\"el Beuzart-Plessis and Binyong Sun for very useful conversations and to Erez Lapid for hinting us to his paper \cite{BL}. 
\normalsize

\section{Notation and basic assumptions}
\subsection{Number fields}
We let $F$ be an algebraic number field. Its set of (non-trivial) places is denoted $S$ and we will write $S_\infty$ for the set of archimedean places. The ring of adeles of $F$ is denoted $\A_F$ or simply by $\A$, the subring of finite adeles is denoted $\A_f$. If $\vv \in S$, $F_\vv $ stands for the topological completion of $F$ with respect to the normalized absolute value, denoted $|\cdot|_\vv $, of $F$. We write $\|\cdot\|_\A:=\prod_{\vv \in S} |\cdot|_\vv $ for the adelic norm.

\subsection{Algebraic groups}
In this paper, $G$ is a connected, reductive linear algebraic group over $F$. 

We assume to have fixed a minimal parabolic $F$-subgroup $P_0$ with Levi decomposition $P_0=L_0N_0$ over $ F $ and let $A_0$ be the maximal $F$-split torus in the center $Z_{L_0}$ of $L_0$. This choice defines the set $ \calP $ of standard parabolic $F$-subgroups $P$ with Levi decomposition $P=L_PN_P$, where $L_P\supseteq L_0$ and $N_P\subseteq N_0$. We let $A_P$ be the maximal $F$-split torus in the center $Z_{L_P}$ of $L_P$, satisfying $A_P\subseteq A_0$. If it is clear from the context, we will also drop the subscript ``$P$''. 

We put $\check\a_P:=X^*(A_P)\otimes_\Z\R$ and $\a_P:=X_*(A_P)\otimes_\Z\R$, where $X^*$ (resp. $X_*$) denotes the group of $F$-rational characters (resp. co-characters). These real Lie algebras are in natural duality to each other. We denote by $\<\cdot,\cdot\>$ the pairing between $\check\a_P$ and $\a_P$ and we let $(\cdot,\cdot)$ be the standard euclidean inner product on $\check\a_{P}\cong\R^{\dim\a_P}$. The inclusion $A_P\subseteq A_0$ (resp. the restriction of characters to $P_0$) defines $\a_P\ra\a_0$ (resp. $\check\a_P\ra\check\a_0$), which gives rise to direct sum decompositions $\a_0=\a_P\oplus\a^P_0$ and $\check\a_0=\check\a_P\oplus\check\a^P_0$. We let $\a^Q_P:=\a_P\cap\a_0^Q$ and $\check\a^Q_P:=\check\a_P\cap\check\a_0^Q$ for parabolic $F$-subgroups $Q$ and $P$. Furthermore, we set $\check\a_{P,\C}:=\check\a_P\otimes_\R\C$ and $\a_{P,\C}:=\a_P\otimes_\R\C$. Then the analogous assertions hold for these complex Lie algebras. 

The group $P$ acts on $N_P$ by the adjoint representation. The weights of this action with respect to the torus $A_P$ are denoted $\Delta(P,A_P)$ and $\rho_P$ denotes the half-sum of these weights, counted with multiplicity. We will not distinguish between $\rho_P$ and its derivative, so we may also view $\rho_P$ as an element of $\a_P$. In particular, $\Delta(P_0,A_0)$ defines a choice of positive $F$-roots of $G$. With respect to this choice, we shall use the notation $\check\a_P^{G+}$ for the open positive Weyl chambers in $\check\a^G_P$.

For $P=LN\in \calP$, let us write $W_L$ for the Weyl group of $L$ with respect to $A_0$, i.e., for the Weyl group attached to the (potentially non-reduced) root system given by the set of all $\pm\alpha$, where $\alpha\in\Delta(P_0,A_0)$ is a positive $F$-root of $G$, which is not in $\Delta(P,A_P)$. Following \cite{moewal}, II.1.7, we will write $W(L)$ for the set of representatives of $W_L\backslash W_G$ of minimal length, for which $wLw^{-1}$ is again the Levi subgroup of a standard parabolic $F$-subgroup of $G$.

\subsection{Locally compact groups}

\subsubsection{Generalities}\label{sect:liegrps} 
We put $G_\infty:=R_{F/\Q}(G)(\R)$, where $R_{F/\Q}$ denotes the restriction of scalars from $F$ to $\Q$. We shall also write $G_\vv :=G(F_\vv )$, $\vv \in S$, whence $G_\infty=\prod_{\vv \in S_\infty} G_\vv $. The analogous notation is used for groups different from $G$. Lie algebras are denoted by the same but lower case gothic letter, e.g., $\g_\infty=\Lie(G_\infty)$ or $\a_{P,\vv }=\Lie(A_{P,\vv })$. Moreover, we write $ \g=\g_\infty $. For every real Lie algebra $ \frakh $, we denote by $\mathcal Z(\frakh)$ the center of the universal enveloping algebra $\U(\frakh)$ of the complex Lie algebra $\frakh_{\C}=\frakh\otimes_\R\C$. In this paper, $\J$ will always stand for an ideal of finite codimension in $\calZ(\g)$. \\\\
We fix a maximal compact subgroup $K_\A$ of $G(\A)$, which is in good position with respect to our choice of standard parabolic subgroups, cf.\ \cite{moewal} I.1.4. It is of the form $K_\A=K_\infty \times K_{\A_f}$, where $K_\infty=\prod_{\vv \in S_\infty} K_\vv $ is a maximal compact subgroup of $G_\infty$ and $K_{\A_f}=\prod_{\vv \notin S_\infty} K_\vv $ is a maximal compact subgroup of $G(\A_f)$, which is hyperspecial at almost all places. For each $\vv \in S$, we choose the Haar measure $ dg_\vv $ on $ G_\vv v $ with respect to which $ \vol(K_\vv )=1 $. The product measures $dg_\infty:=\prod_{\vv \in S_\infty} dg_\vv $, $dg_f:=\prod_{\vv \notin S_\infty} dg_\vv $, and $ dg:=dg_\infty \cdot dg_f $ are then Haar measures on the respective group. Once and for all, we will fix a cofinal sequence $\{K_n\}_{n\in\N}$ (subject to the conditions $K_n\supset K_{n+1}$ and $\bigcap_n K_n = \{\id\}$), forming a neighbourhood base of $\id\in G(\A_f)$ of open compact subgroups $K_n=\prod_{\vv \notin S_\infty} K_{n,\vv }$ of $K_{\A_f}$. \\\\
We denote by $H_P:L_P(\A)\ra\a_{P,\C}$ the standard Harish-Chandra height function, cf. \cite{moewal} I.1.4(4). The group $L_P(\A)^1:=\ker H_P=\bigcap_{\chi\in X^*(L_P)}\ker(\|\chi\|_\A)$ then admits a direct complement $A^\R_P\cong\R_+^{\dim\a_P}$ in $L_P(\A)$ whose Lie algebra is isomorphic to $\a_P$. With respect to $K_\A$, we obtain an extension $H_P:G(\A)\ra\a_{P,\C}$ by setting $H_P(g)=H_P(\ell)$ for $g=\ell n k$. Recall that $K_\A$ has trivial intersection with $A^\R_G$ (but may intersect non-trivially with $A_{G,\infty}$). The same is true for the image of $G(F)$ via the diagonal embedding into $G(\A)$. \\\\
The Lie algebra $\a_P$ of the connected Lie group $A^\R_P$ is viewed as being diagonally embedded into $\a_{P,\infty}$. Moreover, we will denote by $S(\a_{P,\C}):=\bigoplus_{n\in \N}{\rm Sym}^n(\a_{P,\C})$ the symmetric algebra attached to $\a_{P,\C}$, which we will think of as being identified with $\mathcal Z(\a_P)$ as well as with the algebra of polynomials on $\check\a_{P,\C}$, cf.\ \cite{moewal}, I.3.1. Analogously, $S(\check\a^G_{P,\C})$ will denote the symmetric algebra of $\check\a^G_{P,\C}$, which may be viewed as the space of differential operators with constant coefficients on $\check\a^G_{P,\C}$.
 
\subsubsection{Group norms}\label{sect:umg}
Once and for all we fix an embedding $ \iota_G:G\to\GL_N $ defined over $ F $ and define the adelic group norm $ \norm{\spacedcdot}=\norm{\spacedcdot}_{G(\A)}:G(\A)\to\R_{>0} $,
\[ \norm g:=\prod_{\vv \in S}\max_{1\leq i,j\leq N}\left\{\ab{\iota(g_\vv )_{i,j}}_\vv ,\ab{\iota(g_\vv ^{-1})_{i,j}}_\vv \right\}. \]
We note that $\norm g=\norm{g_\infty}\cdot\norm{g_f}$ for all $g=(g_\infty,g_f)\in G(\A)$ and that there exist $ c_0,C_0\in\R_{>0}  $ such that
\begin{equation}\label{eq:ungl}
 \norm g\geq c_0\qquad\text{and}\qquad\norm{gh}\leq C_0\norm g\norm h
 \end{equation}
for all $ g,h\in G(\A) $, cf.\ \cite{moewal}, I.2.2.

\subsection{Locally convex topological vector spaces, direct limits and LF-spaces}\label{sect:LF}
We will use the abbreviation LCTVS to refer to complex, Hausdorff, locally convex topological vector spaces. Moreover, we will use the notion of a strict inductive limit of an increasing sequence of LCTVSs as follows: Let $ (V_n)_{n\in\N} $ be a sequence of LCTVSs such that for each $ n $, $ V_n $ is a closed topological vector subspace of $ V_{n+1} $ (not necessarily a proper one). The strict inductive limit $\lim_{n\ra\infty} V_n$ of the sequence $ (V_n)_{n\in\N} $ is the space $ V:=\bigcup_{n=1}^\infty V_n $ equipped with the finest locally convex topology such that the inclusion maps $\iota_n: V_n\hookrightarrow V $ are continuous. Consequently, a linear map $\phi: V\ra V'$ into an LCTVS $V'$ is continuous if and only if its restriction to each $V_n$, $n\in\N$, is, and a basis of neighbourhoods of $0$ in $V$ is given by the family of subsets 
$$U:={\rm AConv}\left(\bigcup_{n\in\N} U_n\right),$$
where $U_n$ runs through a basis of neighbourhoods of $0$ in each $V_n$ and ``AConv'' denotes the absolute convex hull in $V$. The space $V$ induces on each step $V_n$ its original topology with which $V_n$ becomes a closed subspace of $V$. If each $ V_n $ is complete (resp., barrelled), then so is $V$. If each $ V_n $ is a Fr\'echet space, we say that $ V $ is an {\it LF-space} (``limit Fr\'echet'') with a defining sequence $ (V_n)_{n\in\N} $. In this case, $ V $ is complete, barrelled and bornological, but it is not Fr\'echet unless $ (V_n)_{n\in\N} $ becomes stationary (i.e., $V=V_n$ for some $n$), because it  cannot be Baire (the closed steps $V_n$ have empty interior, as otherwise they would be absorbing and hence all of $V$). We remark that, if each space $V_n$ is finite-dimensional, then $V$ carries its finest locally convex topology.  See \cite{grothendieck_topovs}, Chp.\ 4, Part 1, Prop.\ 1, Cor.\ 1, Prop.\ 2 and Prop.\ 3. \\\\
For a family $ \left\{W_n\right\}_{n\in\N} $ of LCTVSs, we denote by
$$\bigtopdirsum_{n\in\N} W_n$$
the locally convex direct sum of the spaces $W_n$, i.e., the strict inductive limit $V=\lim_{n\ra\infty} \bigoplus^n_{k=1}W_k$. Obviously, if each $W_n$ is Fr\'echet, then $V$ is LF. \\\\
As it is well-known, a closed subspace $W$ of an LF-space $V$ does not need to be an LF-space, cf.\ \cite{grothendieck_LF}, p.\ 89. %However, closed subspaces, which are not ``too far away'' from $V$, allow natural LF-structures, see \cite{helgason00}, Chp.\ IV, Lem.\ 1.13:
 
%\begin{lem}\label{lem:useful-lemma}
%	Let $V= \lim_{n\ra\infty} V_n$ be an $LF$-space and $W \subseteq V$ be a closed subspace. Assume the existence of an open, continuous, surjective, linear map $q: V \to W$, such that $q(V_n) = W \cap V_n$ for all $n$. Then $W$ is again an $LF$-space with a defining sequence $\{W_n\}$ given by $W_n:=W \cap V_n$.
%\end{lem}

\subsection{Representations}\label{sect:rep}
\subsubsection{A few necessary generalities}\label{sect:genreps}
A word on our notions concerning representations (which are the standard ones, to be found, for instance in\ \cite{bowa}), mainly put here to explain their interplay with LF-spaces and in order to fix our specific use of notation, which is tailored to fit the needs when treating representations of adelic groups. To this end, let {\sf G} be a second-countable, locally compact, Hausdorff topological group (e.g., {\sf G} $=G(\A)$, $G_\infty$, $G(\A_f)$ or $G(F_\vv )$). A representation of {\sf G} is a pair $(\pi,V)$, where $V$ is a {\it complete} LCTVS and $\pi: {\sf G} \ra {\rm Aut}_\C(V)$ a group homomorphism such that ${\sf G} \times V \ra V$, given by $({\sf g},v)\mapsto \pi({\sf g})v$, is continuous. If $V$ is barrelled, then the latter is equivalent to $({\sf g},v)\mapsto \pi({\sf g})v$ being separately continuous, i.e., that $\pi({\sf g})$ is a  continuous linear operator $V\ra V$ for all ${\sf g}\in {\sf G}$ and that for each $v\in V$ the orbit map
$$c_v:{\sf G} \ra V,\qquad {\sf g}\mapsto \pi({\sf g})v,$$
is continuous, cf.\ \cite{bourbaki}, VIII, \S 2, Sect.\ 1, Prop.\ 1. We let $\mathcal C_{{\sf G}}$ be the category of ${\sf G}$-representations with morphisms the ${\sf G}$-equivariant continuous linear maps, and {\it irreducibility} and {\it equivalence} of representations is henceforth meant within $\mathcal C_{{\sf G}}$. A representation $(\pi,V)\in \mathcal C_{{\sf G}}$ is {\it unitary}, if $V$ is a Hilbert space and $\pi({\sf g})$ preserves its Hermitian form for all ${\sf g}\in {\sf G}$. In particular, if {\sf G} is compact, then any irreducible representation is finite-dimensional, cf.\ \cite{johnson}, Thm.\ 3.9, and admits an Hermitian form, with respect to which it is unitary. We will use the notions of {\it smoothness} of representations $(\pi,V)\in\mathcal C_{\sf G}$ and vectors $v\in V$, as well as the one of {\it admissibility} of $(\pi,V)\in\mathcal C_{\sf G}$, as in \cite{bowa} 0, \S 2.3 \& \S 2.4, if ${\sf G}$ is a real Lie group with finite component group, resp.\  as in \cite{bowa} X, \S 1.3, if ${\sf G}$ is totally disconnected.  \\\\
It is worth noting that if ${\sf G}$ is totally disconnected, one obtains an equivalence of categories between the subcategory $\mathcal C^{\sf sm, adm}_{{\sf G}}$ of smooth admissible representations in $\mathcal C_{{\sf G}}$ and the category of abstract admissible ${\sf G}$-representations, see\ \cite{bowa} X, \S 5.1 (i.e., using the notion of admissible representation as it is normally done in $p$-adic representation theory, cf.\ \cite{casselman95}, where $V$ is given the discrete topology instead the one of an LCTVS): This equivalence is given by attaching to an abstract admissible ${\sf G}$-representation the LF-topology defined by writing it as the countable increasing union of its invariant subspaces under a suitable cofinal sequence of open compact subgroups, see \cite{bowa} X, \S 1.3 \& \S 5.1. As pointed out in \S \ref{sect:LF}, this LF-space topology is the finest locally convex topology. In particular, as any vector subspace of an LCTVS, which carries its finest locally convex topology, is closed, %cf Bourbaki, TVS, Chp. II, §4, Ex. 6.a.
this equivalence respects irreducibles. 
\subsubsection{Smooth representations}
Remaining in the totally disconnected case, let $\{{\sf K}_n\}_{n\in\N}$ be a decreasing cofinal sequence of compact open subgroups of $ \sf G $, and given a representation $(\pi,V)$ of {\sf G}, let $ E^{{\sf K}_n}:V\to V $ be the standard continuous projection onto $ V^{{\sf K}_n} $:
\[ E^{{\sf K}_n}(v):=\frac1{\mathrm{vol}\,{\sf K}_n}\int_{{\sf K}_n}\pi({\sf k})v\,d{\sf k}, \]
for some fixed Haar measure $d{\sf k}$ on ${\sf K}_n$. We define closed subspaces $ V_n $, $ n\in\N $, of $ V $ as follows:
\begin{equation}\label{eq:013}
V_n:=\begin{cases}
V^{{\sf K}_1},&\text{if }n=1,\\
\ker\left(E^{{\sf K}_{n-1}}\big|_{V^{{\sf K}_n}}\right),&\text{if }n>1.
\end{cases}
\end{equation}
As the restriction $ E^{{\sf K}_n}\big|_{V^{{\sf K}_{n+1}}} $ is a continuous projection onto $ V^{{\sf K}_n} $, one gets
\begin{equation}\label{eq:tddec:0}
	V^{{\sf K}_n}=\bigtoplus_{i=1}^nV_i,\qquad n\in\N,
	\end{equation}
cf.\ \cite{bourTVS}, II, \S 4, Sect.\ 5, Prop.\ 6. We also record the following general lemma. 

\begin{lem}\label{lem:tddec}
	Let  $ \sf G $ be totally disconnected and let $ (\pi,V) $ be a smooth $ \sf G $-representation. Then, 
		\begin{equation}\label{eq:tddec:1}
	V=\bigtoplus_{n\in\N}V_n.
	\end{equation}
Consequently, every subrepresentation $U$ of a smooth $ \sf G $-representation $V$ is smooth. Moreover, if $ V/U $ is a $ \sf G $-representation, then $ V/U $ is a smooth $ \sf G $-representation. 
\end{lem}

\begin{proof}
	If $ (\pi,V) $ is a smooth $ \sf G $-representation, then $ V=\lim_{n}V^{{\sf K}_n}\overset{\eqref{eq:tddec:0}}=\lim_{n}\bigtoplus_{i=1}^nV_i=\bigtoplus_{n\in\N}V_n $, which shows \eqref{eq:tddec:1}. Exploiting \eqref{eq:tddec:0} once more, we get $U=\bigcup_{n=1}^\infty U^{{\sf K}_n}=\bigcup_{n=1}^\infty\bigoplus_{i=1}^n U_i=\bigoplus_{n=1}^\infty U_n$ as vector spaces for every subrepresentation $U$. Since obviously $ U_n\subseteq V_n $ for every $ n\in\N $, \eqref{eq:tddec:1} and \cite{bourTVS}, II, \S 4, Sect.\ 5, Prop.\ 8.(i) imply that
	$ U=\bigtoplus_{n=1}^\infty U_n=\lim_{n}\bigtoplus_{i=1}^nU_i\overset{\eqref{eq:tddec:0}}=\lim_{n}U^{{\sf K}_n} $, hence $ U $ is smooth. Finally, in order to prove that $ V/U $ smooth (if it is a $ \sf G $-representation at all, see Rem.\ \ref{rmk:qtreps} below), it suffices to note that there are the following isomorphisms of LCTVS:
	\begin{equation}\label{eq:008}
	\begin{aligned}
	V/U&\overset{\eqref{eq:tddec:1}}=\Big(\bigtoplus_{n\in\N}V_n\Big)/\Big(\bigtoplus_{n\in\N}U_n\Big)
	\cong\bigtoplus_{n\in\N}V_n/U_n
	=\lim_{n\to\infty}\bigtoplus_{i=1}^nV_i/U_i
	\cong\lim_{n\to\infty}\Big(\bigtoplus_{i=1}^nV_i\Big)/\Big(\bigtoplus_{i=1}^nU_i\Big)\\
	&\overset{\eqref{eq:tddec:0}}=\lim_{n\to\infty}V^{{\sf K}_n}/U^{{\sf K}_n}
	\cong\lim_{n\to\infty}(V/U)^{{\sf K}_n},	
	\end{aligned}
	\end{equation}
Here, the first and second isomorphisms are the canonical ones (cf.\ \cite{bourTVS}, II, \S 4, Sect.\ 5, Prop.\ 8.(ii)), while the last isomorphism is obtained from taking the inverse of the limit of the LCTVS-isomorphisms $ (V/U)^{{\sf K}_n} \ira V^{{\sf K}_n} / U^{{\sf K}_n} $, $v+U\mapsto E^{{\sf K}_n}(v) + U^{{\sf K}_n}$. 
\end{proof}

\begin{rem}\label{rmk:qtreps}
In this generality it is not automatic that the quotient of two (smooth) representations $V/U$ as in Lem.\ \ref{lem:tddec} is a representation. Indeed, it may fail both that the natural map $ {\sf G}\times V/U\to V/U $, $ (g,v+U)\mapsto\pi(g)v+U $ is continuous as well as that the quotient $V/U$ is complete. However, suppose that each quotient $V^{{\sf K}_n}/U^{{\sf K}_n}$ is complete and barrelled. Then, $ V/U $ is also complete and barrelled by \eqref{eq:008} and \cite{grothendieck_topovs}, Chp.\ 4, Part 1, Cor.\ 1 and Prop.\ 3. Thus, the separate continuity of the map $ {\sf G}\times V/U\to V/U $, $ (g,v+U)\mapsto\pi(g)v+U $, implies that $ V/U $ is a representation of $ \sf G $ (see \S\ref{sect:genreps}). 
\end{rem}

Likewise, if ${\sf G}$ is a real Lie group and $(\pi,V)$ a ${\sf G}$-representation, then the $C^\infty$-topology on its space of smooth vectors, i.e., the subspace topology from $C^\infty({\sf G},V)$, coincides with the locally convex topology given by the seminorms $p_{\nu,X}(v):=\nu(\pi(X)v)$, $\nu$ running through the continuous seminorms on $V$ and $X$ through $\mathcal U(\Lie({\sf G}))$, cf.\ \cite{casselman_extensions}, Lem.\ 1.2. Hence, also for Lie groups ${\sf G}$ any subrepresentation of a smooth ${\sf G}$-representation is smooth. Moreover, in the context of strict inductive limits, we obtain

\begin{lem}\label{prop:LFrepsm}
	Let $ V $ be a strict inductive limit of a sequence $ (V_n)_{n\in\N} $ of complete LCTVSs. Let $ (\pi,V) $ be a representation of a real Lie group ${\sf G}$ such that for each $ n $, the closed subspace $ V_n \subseteq V$ is a smooth ${\sf G}$-representation. Then, $ (\pi,V) $ is a smooth ${\sf G}$-representation.
\end{lem}

\begin{proof}
Obviously, every $ v\in V $ is smooth, as the steps $ V_n $ inherit from $ V $ their original topology. It therefore suffices to prove that for every continuous seminorm $ \nu $ on $ V $ and $ X\in\U(\g) $, the seminorm $ p_{\nu,X} $ is continuous, i.e., by \cite{bourTVS}, II, \S4, Sect.\ 4, Prop.\ 5(ii), that the restrictions $ p_{\nu,X}\big|_{V_n}=\nu\big|_{V_n}\circ\pi(X)\big|_{V_n} $ are continuous, which holds by the smoothness of $ V_n $.
\end{proof}

\subsubsection{Representations of $G(\A)$}\label{sect:G(A)reps}
Let now $(\pi,V)$ be a representation of $G(\A)$. We will write $V^{\infty_\A}$ (resp.\ $V^{\infty_\R}$, resp.\  $V^{\infty_f}$) for the space of vectors $v\in V$, for which $c_v$ (resp.\ $c_v|_{G_\infty}$, resp.\ $c_v|_{G(\A_f)}$) is smooth. Then $V^{\infty_\A}$ (as well as $V^{\infty_\R}$ and  $V^{\infty_f}$) is stable under $G(\A)$ and dense in $V$ (by a classical argument of G\aa rding, see \cite{muic_zunar}, (10) and Cor.\ 2(3) for an explicit proof in the current setup).
%for $V^{\infty_\R}$ this is a classical result of G\aa rding, \cite{warner}, Prop. 4.4.1.1, whereas for $V^{\infty_f}$ first identify $V^{\infty_f} = V_{(K_{\A_f})}$ -- the space of $K_{\A_f}$-finite vectors in $V$ -- and then apply the Hahn-Banach theorem). 
We topologize $V^{\infty_\A}$ as the inductive limit $V^{\infty_\A}=\lim_{n\ra\infty} (V^{\infty_\R})^{K_n}$, where $ V^{\infty_\R} $ is equipped with the $ C^\infty $-topology. In this way, $V^{\infty_\A}$ %(resp.\ $V^{\infty_\R}$) 
becomes a representation of $G(\A)$ %(resp.\ $G_\infty$) 
carrying a (potentially strictly) finer topology than the subspace topology coming from $V^{\infty_\R}$ (resp., $ V $). We say that $(\pi,V)$ is a {\it smooth} $G(\A)$-representation, if $V=V^{\infty_\A}$ holds topologically. One checks easily that $(\pi,V)$ is a smooth $G(\A)$-representation if and only if its restrictions to $G_\infty$ and $G(\mathbb A_f)$ are smooth representations. For every $G(\A)$-representation $(\pi, V)$, $V^{\infty_\A}$ is a smooth representation of $G(\A)$. If $V$ is Fr\'echet, then so is $V^{\infty_\R}$, and hence $V^{\infty_\A}$ becomes an LF-space in this case. We have the following basic

\begin{lem}\label{lem:subsquots}
	Let $ (\pi,V) $ be a smooth $ G(\A) $-representation, and let $ U\subseteq V $ be a subrepresentation. Then, $ U $ is a smooth $ G(\A) $-representation. Moreover, if $ V/U $ is a $ G(\A) $-representation (which holds, e.g., if each quotient $V^{K_n}/U^{K_n}$ is complete and barrelled, cf.\ Rem.\ \ref{rmk:qtreps}), then $ V/U $ is a smooth $ G(\A) $-representation.
\end{lem}

\begin{proof} 
	Since $ U $ is a smooth $ G(\A_f) $-representation by Lem.\ \ref{lem:tddec} and is obviously a smooth $ G_\infty $-representation, it is a smooth $ G(\A) $-representation. Next, suppose that $ V/U $ is a representation of $ G(\A) $. To prove that it is a smooth $ G_\infty $-representation, we note that by restricting the natural exact sequence of $G_\infty$-representations $0\ra U \hra V \twoheadrightarrow V/ U\ra 0$ to the subspaces of $G_\infty$-smooth vectors, we obtain a sequence of homomorphisms of $G_\infty$-representations $0\ra U\hra V\rightarrow (V/U)^{\infty_\R}\ra 0$. As $(V/U)^{\infty_\R}\subseteq V/U$, the homomorphism $V\rightarrow (V/ U)^{\infty_\R}$ remains surjective, whence this sequence is exact, too. Thus, the identity map defines a continuous linear bijection $V/U\ra (V/ U)^{\infty_\R}$. As its inverse is obviously continuous, $V/U=(V/U)^{\infty_\R}$ topologically, i.e., $ V/U $ is a smooth $ G_\infty $-representation. Moreover, $ V/U $ is a smooth $ G(\A_f) $-representation by Lem.\ \ref{lem:tddec}, hence it is a smooth $ G(\A) $-representation. 
\end{proof}	

\noindent A $G(\A)$-representation $(\pi,V)$ is called {\it admissible}, if for all irreducible (and hence finite-dimensional, cf.\ \S\ref{sect:genreps}) representations $(\rho,W)$ of $K_\A$ the space $\Hom_{K_\A}(W,V)$ is finite-dimensional. Combining \cite{borel_localement}, \S 2.5 and Prop.\ 3.6, one verifies that this definition -- intrinsic to the notion of $G(\A)$-representations -- is equivalent to the one in \cite{bowa}, XII, \S 1.4: For all $n\in \N$ and for all irreducible representations $(\rho_\infty,W_\infty)$ of $K_\infty$, the space $\Hom_{K_\infty}(W_\infty,V^{K_n})$ is finite-dimensional. If $V$ is admissible, $V^{\infty_\A}_{(K_\infty)}=V^{\infty_f}_{(K_\infty)}$ and $V^{\infty_\A}$ is an admissible $G(\A)$-representation with the same $K_\A$-isotypic components. \\\\
We also recall the notion of a $(\g_\infty,K_\infty, G(\A_f))$-{\it module}: That is a $(\g_\infty,K_\infty)$-module $V_0$, which carries a smooth linear action of $G(\A_f)$ (i.e., each vector $v\in V_0$ has a smooth orbit map $c_v:G(\A_f)\ra V_0$), which commutes with the actions of $\g_\infty$ and $K_\infty$, see \cite{bowa}, XII, \S 2.2. A $(\g_\infty,K_\infty, G(\A_f))$-module is called {\it admissible}, if for all irreducible (and hence finite-dimensional) representations $(\rho,W)$ of $K_\A$ the space $\Hom_{K_\A}(W,V_0)$ is finite-dimensional.
The assignment $V\mapsto V^{\infty_\A}_{(K_\infty)}$ defines a functor from the category of $G(\A)$-representations to the category of $(\g_\infty,K_\infty, G(\A_f))$-modules \cite{bowa}, XII, \S 1.3 \& \S 2.4. Obviously, a smooth representation $(\pi,V)$ of $G(\A)$ is admissible, if and only if its underlying $(\g_\infty,K_\infty, G(\A_f))$-module $V^{\infty_\A}_{(K_\infty)}=V_{(K_\infty)}$ is. The following lemma will be useful later:

\begin{lem}\label{lem:dense_subm}
	Let $ (\pi,V) $ be a smooth $ G(\A) $-representation. Let $ V_0 $ be an admissible $ (\g_\infty,K_\infty,G(\A_f)) $-submodule of $ V_{(K_\infty)} $ that is dense in $ V $. Then, $V_{(K_\infty)} = V_0$ and so $ (\pi,V) $ is admissible.
\end{lem}

\begin{proof}
	 Let us prove that for all irreducible (hence finite-dimensional) representations $(\rho_\infty,W_\infty)$ of ${K_\infty} $ and $ n\in\N $, we have the equality of $ \rho_\infty $-isotypic components
	\begin{equation}\label{eq:009}
	V^{K_n}(\rho_\infty)=V^{K_n}_0(\rho_\infty).
	\end{equation}
	To this end recall the operator $ E_{\rho_\infty}^{K_n}:V\to V $,
	\begin{equation}\label{eq:Erhogl}
	 E_{\rho_\infty}^{K_n}(v):=\frac{d(\rho_\infty)}{\mathrm{vol}\, K_n}\,\int_{K_\infty}\int_{K_n}\overline{\xi_{\rho_\infty}(k_\infty)} \ R(k_\infty k_f)v \ d k_f \ dk_\infty, 
	 \end{equation}
	where $ d(\rho_\infty) $ and $ \xi_{\rho_\infty} $ are, respectively, the degree and character of $\rho_\infty $. Then $ E_{\rho_\infty}^{K_n}$ is a continuous projection onto $ V^{K_n}(\rho_\infty) $ and restricts to a projection of $ V_0 $ onto $ V^{K_n}_0(\rho_\infty) $, cf.\ \cite{muic_zunar}, Lem.\ 19. Thus, we have
	\[ V^{K_n}(\rho_\infty)=E_{\rho_\infty}^{K_n}(V)=E_{\rho_\infty}^{K_n}\left(\Cl_{V}(V_0)\right)\subseteq\Cl_{V}\left(E_{\rho_\infty}^{K_n}(V_0)\right)=\Cl_{V}\left(V^{K_n}_0(\rho_\infty)\right)=V^{K_n}_0(\rho_\infty), \]
	where the inclusion holds by the continuity of $ E_{\rho_\infty}^{K_n} $, and the last equality is true because $ V^{K_n}_0(\rho) $ is finite-dimensional. Since the reverse inclusion is obvious, this proves \eqref{eq:009}. Therefore, 
	\[ V_{(K_\infty)}=\bigcup_n\bigoplus_{\rho_\infty} V^{K_n}(\rho_\infty)\overset{\eqref{eq:009}}=\bigcup_n\bigoplus_{\rho_\infty} V^{K_n}_0(\rho_\infty)=V_0, \]
which shows the first claim. As $V$ is smooth, by our above observation, it is admissible, if and only if its underlying $(\g_\infty,K_\infty, G(\A_f))$-module $V^{\infty_\A}_{(K_\infty)}=V_{(K_\infty)}$ is. However, by what we have just proved, $V_{(K_\infty)}=V_0$, and so admissibility of $V$ follows from the assumption that $V_0$ is admissible.
\end{proof}

\subsubsection{Casselman-Wallach representations}\label{subsec:cass_wall}

Following \cite{casselman_extensions,wallachII}, we say that a $G_\infty$-representation $(\pi,V)$ on a Fr\'echet space $V$ is {\it of moderate growth}, if for every continuous seminorm $p$ on $V$, there exist a real number $d = d_p$ and a continuous seminorm $q = q_p$ on $V$ such that
\[
p\left( \pi(g_\infty)v \right) \le \norm{g_\infty}^{d_p} q_p(v) \qquad \text{ for all } g_\infty \in G_\infty, v \in V.
\]
Here, $\norm{g_\infty}$ is the group norm from \S \ref{sect:umg} as applied to $g_\infty=(g_\infty,\id)\in G(\A)$. We refer to \cite{wallachII}, \S 11.5 for the basic properties of smooth representations of moderate growth. We will call a smooth admissible $G_\infty$-representation $(\pi,V)$ of moderate growth a {\it Casselman-Wallach representation of $ G_\infty $}, if its underlying $(\g_\infty,K_\infty)$-module $V_{(K_\infty)}$ is finitely generated (cf.\ \cite{wallachII} \S 11.6.8) or equivalently $ \calZ(\g) $-finite (cf.\ \cite{vogan}, Cor.\ 5.4.16).\\\\ 
A smooth representation $ (\pi,V) $ of $ G(\A) $ such that for every $ n\in\N $, $ V^{K_n} $ is a Casselman-Wallach representation of $ G_\infty $, will be called a \textit{Casselman-Wallach representation of $ G(\A) $}. By the above said, every Casselman-Wallach representation $ (\pi,V) $ of $ G(\A) $ is admissible and $V$ is an LF-space. Moreover, the following well-known, classical results of Wallach translate into the ``global'' setup: 

\begin{lem}\label{lem:cass_wall_reps}
	\begin{enumerate}
		\item\label{lem:cass_wall_reps:1} Let $ (\pi,V) $ be a Casselman-Wallach representation of $ G(\A) $. If $ U $ is a subrepresentation of $ V $, then $ U $ and $ V/U $ are Casselman-Wallach representations of $ G(\A) $.
		\item\label{lem:cass_wall_reps:2} Two Casselman-Wallach representations $ (\pi,V) $ and $ (\sigma,W) $ of $ G(\A) $ are isomorphic if and only if the underlying $ (\g_\infty,K_\infty,G(\A_f)) $-modules $ V_{(K_\infty)} $ and $ W_{(K_\infty)} $ are isomorphic.
	\end{enumerate}
\end{lem}

\begin{proof}
(1) The analogous assertion about Casselman-Wallach representations of $ G_\infty $ follows from \cite{wallachII}, Lem.\ 11.5.2 and implies the claims about $ U $ and $ V/U $ using Lem.\ \ref{lem:subsquots} together with the following elementary fact: For every $ n\in\N $, the assignment $ v+U\mapsto E^{K_n}v+U^{K_n} $ defines a $ G_\infty $-equivariant LCTVS-isomorphism $ (V/U)^{K_n}\ira V^{K_n}/U^{K_n} $, cf.\ \eqref{eq:008}.

(2) It suffices to prove that a $ (\g_\infty,K_\infty,G(\A_f)) $-module isomorphism $ T_0:V_{(K_\infty)}\ira W_{(K_\infty)} $ extends to an isomorphism of $ G(\A) $-representations $ T:V\ira W $. By a celebrated result due to Casselman and Wallach, cf.\ see \cite{wallachII}, Thm.\ 11.6.7(2), for every $ n\in\N $ we can extend the $ (\g_\infty,K_\infty) $-module isomorphism $ T_0\big|_{V_{(K_\infty)}^{K_n}}:V_{(K_\infty)}^{K_n}\ira W_{(K_\infty)}^{K_n} $ in a unique way to an isomorphism of $ G_\infty $-representations $ T_n:V^{K_n}\ira W^{K_n} $. By the uniqueness of this extension, $ T_{n+1} $ extends $ T_n $. Hence, going over to the direct limit over $ n\in\N $, we obtain a well-defined isomorphism of $ G_\infty $-representations $ T:V\to W $, which extends $ T_0 $ (and is uniquely determined by this property). Since $ T_0 $ is $ G(\A_f) $-equivariant, so is $ T $ by continuity, hence $ T $ is an isomorphism of $ G(\A) $-representations $T:V\ira W $.
\end{proof}

\section{The LF-space of smooth-automorphic forms}

\subsection{Spaces of functions of uniform moderate growth}\label{sect:smumg}
Let $ C^\infty(G(F)\backslash G(\A)) $ be the space of smooth functions $ G(\A)\to\C $, which are invariant under multiplication by $ G(F)$ from the left. For $ n,d\in\N $ we define the space $ C^\infty_{umg,d}(G(F)\backslash G(\A))^{K_n} $ to consist of the right-$ K_n $-invariant functions $ f\in C^\infty(G(F)\backslash G(\A)) $ that satisfy the following uniform moderate growth condition: For every $ X\in\U(\g) $, 
\begin{equation}\label{eq:010}
p_{d,X}(f):=\sup_{g\in G(\A)}\ab{(Xf)(g)}\,\norm g^{-d}<\infty.
\end{equation}
We note that the space $ C^\infty_{umg}(G(F)\backslash G(\A)):=\bigcup_{n,d}C^\infty_{umg,d}(G(F)\backslash G(\A))^{K_n} $ does not depend on the choice of $ \iota_G$, used, in order to define $\norm g$, cf.\ \S \ref{sect:umg}, although the individual spaces $ C^\infty_{umg,d}(G(F)\backslash G(\A))^{K_n} $ do. It follows from \eqref{eq:ungl} that with this definition 
\begin{equation}\label{eq:incl}
 C^\infty_{umg,d}(G(F)\backslash G(\A))^{K_n} \subseteq C^\infty_{umg,d+1}(G(F)\backslash G(\A))^{K_n},
 \end{equation}
for all $ n,d\in\N $. We equip $ C^\infty_{umg,d}(G(F)\backslash G(\A))^{K_n} $ with the locally convex topology, defined by the seminorms $p_{d,X}$, $ X\in\U(\g) $. Let us point out that even with this natural topology, \eqref{eq:incl} does not necessarily define an embedding of closed subspaces. %In fact, $ C^\infty_{umg,d}(G(F)\backslash G(\A))^{K_n}$ may even be dense in $ C^\infty_{umg,d+1}(G(F)\backslash G(\A))^{K_n}$.

For fixed $ n $ and $ d $, it will be useful to relate the space $ C^\infty_{umg,d}(G(F)\backslash G(\A))^{K_n} $ to the spaces of smooth functions of uniform moderate growth on $ G_\infty $, which is done in the following standard way: We recall that there exists a finite set $ C\subseteq G(\A_f) $ such that $ G(\A)=\bigsqcup_{c\in C}G(F)(G_\infty \cdot c\cdot K_n) $, \cite{borel1963}, Thm.\ 5.1. For every $ c\in C $, the congruence subgroup $ \Gamma_{c,K_n}:=cK_nc^{-1}\cap G(F) $ of $ G(F) $ embeds into $ G_\infty $ as a discrete subgroup of finite covolume modulo $A^\R_G$. Let $ C^\infty_{umg,d}(\Gamma_{c,K_n}\backslash G_\infty) $ denote the space of smooth, left-$ \Gamma_{c,K_n} $-invariant functions $ f:G_\infty\to\C $ such that
\[ p_{\infty,d,X}(f):=\sup_{g_\infty\in G_\infty}\ab{(Xf)(g_\infty)}\,\norm{g_\infty}^{-d}<\infty,\qquad X\in\U(\g). \] Again, we equip $ C^\infty_{umg,d}(\Gamma_{c,K_n}\backslash G_\infty) $ with the locally convex topology defined by the seminorms $p_{\infty,d,X}$, $ X\in\U(\g) $. With this setup in place, we obtain

\begin{prop}\label{prop:dict_umg}
	The assignment
	$ f\mapsto\left(f(\rule{3mm}{0.15mm} \cdot c)\right)_{c\in C} $
	defines an isomorphism of Fr\'echet spaces
	\[ C^\infty_{umg,d}(G(F)\backslash G(\A))^{K_n}\cong\bigtoplus_{c\in C}C^\infty_{umg,d}(\Gamma_{c,K_n}\backslash G_\infty), \]
	which restricts to an isomorphism of their closed subspaces of $A^\R_G$-invariant functions.
\end{prop}
 
\begin{proof}
	It follows from \cite{wallach_smooth}, Lemma 2.7, that the spaces $C^\infty_{umg,d}(\Gamma_{c,K_n}\backslash G_\infty)$ (and hence their finite direct sum) are Fr\'echet. %See also proof of Thm.\ 1.16 in \cite{casselman_schwartz} (together with Prop.\ 1.7, ibidem)
The map $ f\mapsto\left(f(\rule{3mm}{0.15mm} \cdot c)\right)_{c\in C} $ is obviously a continuous bijection, whose inverse is given by the following construction: For a $ (f_c)_{c\in C}\in\bigtoplus_{c\in C}C^\infty_{umg,d}(\Gamma_{c,K_n}\backslash G_\infty) $, let $ f\in C^\infty(G(F)\backslash G(\A))^{K_n} $ be defined by
	\[ f(\delta g_\infty c\ell):=f_c\left(g_\infty\right),\qquad \delta\in G(F),\ g_\infty\in G_\infty,\ c\in C,\ \ell\in K_n. \]
	In order to show that this inverse is continuous (only from which it will follow that $C^\infty_{umg,d}(G(F)\backslash G(\A))^{K_n}$ is also a Fr\'echet space), fix a Siegel set $ \mathfrak S $ in $ G(\A) $ as in \cite{moewal}, p.\ 20. Since the projection of $ \mathfrak S $ on $ G(\A_f) $ is compact, $ \mathfrak S\subseteq \Delta G_\infty CK_n $ for some finite set $ \Delta\subseteq G(F) $. Moreover, since $ G(\A)=G(F)\mathfrak S $, by \cite{moewal}, I.2.2(vii) there exists $ M\in\R_{>0} $ such that for every $ X\in\U(\g) $,
	\begin{align*}
	p_{d,X}(f)&\overset{\phantom{\eqref{eq:ungl}}}\leq M\,\sup_{g\in\mathfrak S}\ab{(Xf)(g)}\norm g^{-d}\\
	&\overset{\phantom{\eqref{eq:ungl}}}\leq M\,\sup_{\substack{\delta\in \Delta,\,g_\infty\in G_\infty\\c\in C,\,\ell\in K_n}}\ab{(Xf)(\delta g_\infty c\ell)}\,\norm{\delta g_\infty c\ell}^{-d}\\
	&\overset{\eqref{eq:ungl}}\leq M\,\sup_{\substack{\delta\in \Delta,\,g_\infty\in G_\infty\\c\in C,\,\ell\in K_n}}\ab{(Xf_c)\left(g_\infty\right)}\,C_0^{3d}\,\norm{\delta}^{d}\,\norm{g_\infty}^{-d}\,\norm{c}^d\,\norm{\ell}^{d}\\
	&\overset{\phantom{\eqref{eq:ungl}}}\leq M\,C_0^{3d}\,\left(\max_{\delta\in \Delta}\norm{\delta}^d\right)\left(\max_{\ell\in K_n}\norm \ell^{d}\right)\left(\max_{c\in C}\norm c^d p_{\infty,d,X}(f_c)\right).
	\end{align*}
	Thus, $ f\mapsto\left(f(\rule{3mm}{0.15mm} \cdot c)\right)_{c\in C} $ is an isomorphism of LCTVS and (therefore) $C^\infty_{umg,d}(G(F)\backslash G(\A))^{K_n}$ is a Fr\'echet space. The remaining claim about the closed subspaces of $A^\R_G$-invariant functions is obvious.
\end{proof} 

\begin{warnrem}\label{rem:warning}
Let us write $C^\infty_{umg,d}(G(F)\backslash G(\A)):=\bigcup_{n} C^\infty_{umg,d}(G(F)\backslash G(\A))^{K_n}$
for the space of smooth, left-$G(F)$-invariant functions, which satisfy \eqref{eq:010} for a given $d\in\N$, and let a subscript ``$(\calZ(\g))$'' denote the subspace of $\calZ(\g)$-finite vectors. We take the opportunity to point out that -- although {\it each} of the corresponding claims can be found in the sparse literature on smooth-automorphic forms -- actually {\it none} of the following three spaces is Fr\'echet, when equipped with the seminorms $p_{d,X}$, $X\in\U(\g)$: 
$$C^\infty_{umg,d}(G(F)\backslash G(\A))\supseteq C^\infty_{umg,d}(G(F)\backslash G(\A))_{(\calZ(\g))}\supseteq C^\infty_{umg,d}(G(F)\backslash G(\A))^{K_n}_{(\calZ(\g))}. $$ % In this order: \cite{erez_book}, p.241, \cite{krs}, p.\ 496, \cite{cogdell_fields}, \S3.2
Indeed, one may easily verify that none of the above spaces is complete. \\\\
In our eyes, the resulting topological intricacies (as well as the resulting issues in providing completely accurate references) make it necessary to lay down some details on the topological structure of the (in this paper yet to be defined) space of smooth-automorphic forms and the canonical action of $G(\A)$ by right translations, which we will do in the remainder of this section. The reader, who is familiar with these functional-analytic subtleties, may skip them and head on to \S\ref{sect:smautreps}.
\end{warnrem}

\subsection{Smooth-automorphic forms and absolute bounds of exponents of growth}

\subsubsection{Spaces of smooth-automorphic forms}
Recall our arbitrary, but fixed ideal $\J \lhd\calZ(\g)$ of finite codimension. For each $n\in \N$, we let 
$$\calA^\infty_d(G)^{K_n,\J^n}:=\{\varphi\in C^\infty_{umg,d}(G(F)\backslash G(\A))^{K_n} \ | \ \J^n\varphi=0\}$$
be the (closed, and hence Fr\'echet, cf.\ Prop.\ \ref{prop:dict_umg}) subspace of $C^\infty_{umg,d}(G(F)\backslash G(\A))^{K_n}$ of functions, which are annihilated by the ideal $\J^n$. We obtain

\begin{prop}\label{prop:sm_inf}
Acted upon by right translation $R_\infty$, the spaces $\calA^\infty_d(G)^{K_n,\J^n}$ become Casselman-Wallach representations of $ G_\infty $.
\end{prop}
\begin{proof}
The arguments of \cite{wallach_smooth}, \S 2.5 and Lem.\ 2.7 imply that $C^\infty_{umg,d}(\Gamma_{c,K_n}\backslash G_\infty)$ is a smooth $G_\infty$-representation by right translation of moderate growth. Hence, so is $C^\infty_{umg,d}(G(F)\backslash G(\A))^{K_n}$ by virtue of Prop.\ \ref{prop:dict_umg} and so is the $G_\infty$-stable, closed subspace $\calA^\infty_d(G)^{K_n,\J^n}$, cf.\ \cite{wallachII}, Lem.\ 11.5.2. Moreover, the $ (\g_\infty,K_\infty) $-module $ \left(\calA^\infty_d(G)^{K_n,\J^n}\right)_{(K_\infty)} $ is admissible by \cite[\S 4.3.(i)]{bojac} and is obviously $ \calZ(\g) $-finite. This shows the claim. 
\end{proof}

\begin{defn}\label{def:smautf}
A {\it smooth-automorphic form} is an element of one of the spaces 
$$\calA^\infty_\J(G):=\bigcup_{n,d} \calA^\infty_d(G)^{K_n,\J^n}, $$
where $ \J $ runs over the ideals of finite codimension in $ \calZ(\g) $.
\end{defn}

\begin{rem}
Our definition is obviously compatible with the usual one: A smooth function $f: G(\A)\ra\C$ is usually defined to be a smooth-automorphic form, if it is left-$G(F)$-invariant, right-$K_{\A_f}$-finite\footnote{This condition is recalled here only for the sake of matching the way smooth-automorphic forms are usually defined: Let us take the opportunity to point out that smoothness of $\varphi$ makes the condition of being right-$K_{\A_f}$-finite superfluous.}, annihilated by an ideal $\J \lhd\calZ(\g)$ of finite codimension and of uniform moderate growth. See \cite{wallach_smooth}, \S 6.1, where to our knowledge the notion of smooth-automorphic forms has first been introduced (in the context of real reductive groups and their arithmetic subgroups) or \cite{cogdell_fields}, \S 2.3 (for $G=\GL_n$). 
\end{rem}

\noindent We want to consider $\calA^\infty_\J(G)$ as a representation of $G(\A)$ under right translation, whence we have to specify a locally convex topology on $\calA^\infty_\J(G)$, making it into a complete LCTVS, cf.\ \S\ref{sect:rep}. In a very first try, it is tempting to put an ordering on the set of tuples $(d,n)\in\N\times\N$, with defining condition that $\calA^\infty_d(G)^{K_n,\J^n}\subset \calA^\infty_{d+1}(G)^{K_{n+1},\J^{n+1}}$ and then equip $\calA^\infty_\J(G)$ with the inductive limit topology given by the natural, continuous inclusions.\\\\ However, estimating the respective seminorms in question, it is a priori not clear at all that the locally convex topology on $\calA^\infty_{d+1}(G)^{K_{n+1},\J^{n+1}}$ (defined by the seminorms $p_{d+1,X}$, $X\in\U(\g)$), induces the original topology on $\calA^\infty_d(G)^{K_n,\J^n}$ (defined by the seminorms $p_{d,X}$, $X\in\U(\g)$), nor if $\calA^\infty_d(G)^{K_n,\J^n}$ is closed in $\calA^\infty_{d+1}(G)^{K_{n+1},\J^{n+1}}$, see \S\ref{sect:smumg}. Whence, it is a priori even unclear whether the latter inductive limit topology yields a Hausdorff space at all, and even less, if the resulting space is a complete LCTVS. \\\\ 
In order to overcome this problem, we shall make use of the following general result:  

\begin{prop}\label{prop:same_d}
Let $\J \lhd\calZ(\g)$ be an arbitrary, but fixed ideal of finite codimension. Then there exists $d\in\N$ such that 
$$\calA^\infty_\J(G)\subset C^\infty_{umg,d}(G(F)\backslash G(\A)).$$
\end{prop} 
In other words, having fixed $\J$ (or, equivalently, the string of ideals $\{\J^n\}_{n\in\N}$), then there exists an exponent $d$, such that all smooth-automorphic forms in $\calA^\infty_\J(G)$ satisfy \eqref{eq:010} for {\it the same} such $d$. In order to prove Prop.\ \ref{prop:same_d} we shall need the following preparatory considerations:

\subsubsection{Growth conditions for constant terms}\label{sub:007}
Let $P=LN\in\calP$. Let us introduce coordinates $ \a_P\ira\R^{n_P} $, $ \lambda\mapsto\underline \lambda $, and the following multi-index notation:
\[ \lambda^{\alpha}:=\underline \lambda_1^{\alpha_1}\cdots \underline \lambda_{n_P}^{\alpha_{n_P}},\qquad \lambda\in\a_P,\ \alpha=(\alpha_1,\ldots,\alpha_{n_P})\in\Z^{n_P}_{\geq0}. \]
For a multi-index $ \alpha $ as above, we write $|\alpha|:=\sum_{i=1}^{n_P}\alpha_i $.\\\\
For every subset $ \Lambda $ of $\check\a_{P,\C}$ and $ N\in\Z_{\geq0} $, let $ \mathfrak{P}_{A_P^\R}(\Lambda,N) $ be the space of functions  $ \phi:A_P^\R\to\C $ of the form
\[ \phi(a)=\sum_{\lambda\in\Lambda'}e^{\langle\lambda,H_P(a)\rangle}\sum_{|\alpha|\leq N}c_{\lambda,\alpha}H_P(a)^\alpha, \]
where $ \Lambda' $ runs through the finite subsets of $ \Lambda $, and $ c_{\lambda,\alpha}\in\C $ (``polynomial exponential functions''). For a continuous, left-$G(F)$-invariant function $f: G(\A)\ra\C$, we write as usual \[ f_P(g):=\int_{N(F)\backslash N(\A)}f(ng)\,dn, \]
for the constant term along $P$, cf.\ \cite{moewal}, I.2.6.

\begin{lem}\label{lem:002}
For every $P\in\calP$, let $ \Lambda_P $ be a finite subset of $\check\a_{P,\C}$. Then there exists an $ r\in\R_{>0} $ such that 
$$\{f\in C^\infty_{umg}(G(F)\backslash G(\A))\ | \ \forall P\in\calP\quad\exists N\in\Z_{\geq0}\quad\forall g\in G(\A)\quad f_P(\rule{3mm}{0.15mm} \cdot g)\in\mathfrak{P}_{A_P^\R}(\Lambda_P,N)\}$$	
is contained in $C^\infty_{umg,r}(G(F)\backslash G(\A))$.	
\end{lem}

\begin{proof}
	This follows from the proof of \cite{BL}, Lem.\ 6.12, noting that the constant $ r $ whose existence is proved in said lemma depends, in the notation of {\it loc.\ cit.}, only on $ \dim \a_0 $ and on the constant $ \max_{P}\max_{(\lambda_1,\ldots,\lambda_{n_P})\in\Lambda_P}\max_{i=1}^{n_P}\norm{\lambda_i}$, where $ \norm{\,\cdot\,} $ is a fixed norm on $X_P $, and does not depend on the numbers $ n_P $.
\end{proof}

\subsubsection{Ideals of finite codimension}\label{sub:008}
\noindent Recall the symmetric algebra $S(\a_{P,\C})\ira\calZ(\a_P)$ of $\a_{P,\C}$, cf.\ \S\ref{sect:liegrps}, which we identify with the algebra of polynomials on $\check\a_{P,\C}$. Every ideal $ \mathcal I \lhd  \calZ(\a_P) $ of finite codimension contains an ideal of the form
\[ \calZ(\a_P;\Lambda,N):=\left\{Y\in\calZ(\a_P) \ | \ Y\text{ vanishes of order }\geq N\text{ in each }\lambda\in\Lambda\right\}, \]
where $ \Lambda $ is a finite subset of $ \check\a_{P,\C} $ and $ N\in\N $, see \cite{moewal}, I.3.1. Next note that for such $ \mathcal I $, $ \Lambda $ and $ N $, and for every $ k\in\N $ there exists $ M\in\N $ (depending on $ \mathcal I $, $ \Lambda $, $ N $, and $k$), such that
$$\bigcap_{\lambda\in\Lambda}\calZ\left(\a_P;\lambda,N\right)^M\subseteq \left(\bigcap_{\lambda\in\Lambda}\calZ\left(\a_P;\lambda,N\right)\right)^k.$$
Indeed, the existence of such an exponent $M\in \N$ is a corollary of the following general

\begin{lem}\label{lem:003}
	Let $ R $ be a commutative Noetherian ring. Let $ n,k\in\N $, and let $ \mathcal I_1,\ldots,\mathcal I_n $ be ideals in $ R $. Then, there exists $ M\in\N $ such that
	\[ \bigcap_{i=1}^n\mathcal I_i^M\subseteq \left(\bigcap_{i=1}^n\mathcal I_i\right)^k. \]
\end{lem}
\begin{proof}
	Lacking a good reference, we sketch the proof. Arguing inductively, let $ n=2 $. Then, by the lemma of Artin-Rees \cite[\S 2.3, Lem.\ 1]{bosch}, there exists an $ l\in\Z_{\geq0} $ such that $ \mathcal I_1^{k+l}\cap\mathcal I_2^k=\mathcal I_1^k\left(\mathcal I_1^l\cap\mathcal I_2^k\right) $. Hence, setting $M=k+l$, in particular $ \mathcal I_1^{M}\cap\mathcal I_2^{M}\subseteq\mathcal I_1^k\mathcal I_2^k=(\mathcal I_1\mathcal I_2)^k\subseteq(\mathcal I_1\cap\mathcal I_2)^k $ holds. The induction-step is obvious.
\end{proof}
Consequently, 
\[ \mathcal I^k\rhd\calZ\left(\a_P;\Lambda,N\right)^k=\left(\bigcap_{\lambda\in\Lambda}\calZ\left(\a_P;\lambda,N\right)\right)^k\supseteq\bigcap_{\lambda\in\Lambda}\calZ\left(\a_P;\lambda,N\right)^M=\bigcap_{\lambda\in\Lambda}\calZ\left(\a_P;\lambda,MN\right)=\calZ\left(\a_P;\Lambda,MN\right). \]
We are now ready to give the 

\begin{proof}[Proof of Prop.\ \ref{prop:same_d}]
By Lem.\ \ref{lem:002}, it suffices to prove the following claim: For every $ P\in\calP $ there exists a finite set $ \Lambda_P\subseteq\check\a_{P,\C}$ with the following property: For every $\varphi\in\calA_\J^\infty(G) $ there exists $ N\in\Z_{\geq0}$ such that
\begin{equation}\label{eq:tP} 
\varphi_P(\rule{3mm}{0.15mm} \cdot g)\in\mathfrak P_{A_P^\R}(\Lambda_P,N),\qquad g\in G(\A). 
\end{equation}
First, we recall that, denoting by $ \mathcal L $ the left action of $ \U(\g) $ on $ C^\infty(G(\A)) $, we have
\[ Yf=\mathcal L(Y^\#)f,\qquad Y\in\calZ(\g),\ f\in C^\infty(G(\A)), \]
where $ (\cdot)^\# $ is the unique anti-automorphism of $ \U(\g) $ such that $ 1^\#=1 $ and $ X^\#=-X $ for all $ X\in\g_\infty $.

Now, let $ P=LN\in\mathcal P$. Then, denoting the Harish-Chandra homomorphism by $ \nu:\calZ(\g)\to\calZ(\l) $, we get $Y\in\nu(Y)+\U(\g)\n_\C$ for all $Y\in\calZ(\g)$, and hence 
\begin{equation}\label{eq:005}
(Yf)_P=Yf_P=\mathcal L(Y^\#)f_P=\mathcal L(\nu(Y^\#))f_P=\nu(Y^\#)^\#f_P
\end{equation}
for all $ Y\in\calZ(\g) $ and $ f\in C^\infty_{umg}(G(F)\backslash G(\A))  $. As $ \calZ(\l) $ is a finitely generated $ \nu(\calZ(\g)) $-module, $\J_\l:=\calZ(\l)\nu(\J^\#)^\# $ is an ideal of finite codimension in $ \calZ(\l) $. 

Let $ \varphi\in\calA_\J^\infty(G) $, and let $ k\in\N $ be such that $\J^k\varphi=0 $. For every $ g\in G(\A) $, the function $ \varphi(\rule{3mm}{0.15mm} \cdot g):G(\A)\to\C $ is also annihilated by $\J^k $, hence by \eqref{eq:005} the function $ \varphi_P(\rule{3mm}{0.15mm} \cdot g):G(\A)\to\C $ is annihilated by $ \calZ(\l)\nu((\J^k)^\#)^\#=\J_\l^k $. Since the algebra $ \calZ(\a_P) $ embeds canonically into $ \calZ(\l) $, it follows that the function $\varphi_P(\rule{3mm}{0.15mm} \cdot g):A_P^\R\to\C $ is annihilated by the following ideals in $ \calZ(\a_P) $:
\[\J_\l^k\cap\calZ(\a_P)\supseteq(\J_\l\cap\calZ(\a_P))^k\supseteq\calZ(\a_P;\Lambda_P,M) \]
for some $ M\in\N $, where $ \Lambda_P $ is chosen such that $\J_\l\cap\calZ(\a_P)\supseteq\calZ(\a_P;\Lambda_P,N)$ for some $ N\in\N $, see \S\ref{sub:008} above. But by \cite{moewal}, I.3.1,
\[ \left\{f\in C^\infty(A_P^\R):\calZ(\a_P;\Lambda_P,M)f=0\right\}\subseteq\mathfrak P_{A_P^\R}\left(\Lambda_P,M\right). \]
Therefore, 
\[ \varphi_P(\rule{3mm}{0.15mm} \cdot g)\in\mathfrak P_{A_P^\R}\left(\Lambda_P,M\right),\qquad g\in G(\A). \] 
Since $ \Lambda_P $ depends only on $\J $ and $ P $, this proves the above claim, cf.\ \eqref{eq:tP}, and thus Prop.\ \ref{prop:same_d}.	
\end{proof}

\subsubsection{The LF-space topology on $\calA_\J^\infty(G)$}\label{sect:LF-autom}
As provided by Prop.\ \ref{prop:same_d}, there exists a smallest $d=d_0\in\N$ such that $\calA^\infty_\J(G)\subset C^\infty_{umg,d}(G(F)\backslash G(\A)).$ We will henceforth fix such an exponent $d$. Recall that for each $n\in\N$, $\calA^\infty_d(G)^{K_n,\J^n}$ has been given the structure of a Fr\'echet space. Having fixed $d$, we obtain canonical continuous inclusions $\iota_n: \calA^\infty_d(G)^{K_n,\J^n}\hra \calA^\infty_d(G)^{K_{n+1},\J^{n+1}}$ with closed images and we will hence equip the space of smooth-automorphic forms $\calA_\J^\infty(G)$ with its natural LF-space topology given by 
\begin{equation}\label{eq:LF_smaut}
\calA_\J^\infty(G)=\lim_{n\ra\infty}\calA^\infty_d(G)^{K_n,\J^n}.
\end{equation}
As remarked in \S \ref{sect:LF}, this makes $\calA_\J^\infty(G)$ into a complete, barrelled, bornological LCTVS (which, as we recall, for us includes the property of being Hausdorff). It is easy to see that if $(\varphi_i)_{i\in I}$ is a convergent net in $\calA_\J^\infty(G)$ with limit $\varphi$, then the net of complex numbers $((X\varphi_i)(g))_{i\in I}$ converges in $\C$ to $(X\varphi)(g)$ for every $X\in\U(\g)$ and $g\in G(\A)$. In particular, every convergent net of smooth-automorphic forms is pointwise convergent everywhere.

\begin{rem}\label{rem:018}
Our particular choice of an exponent of growth $d$ is only made for the convenience of normalization and can be replaced by any other integer $d'\geq d$, without change of topology, as the following argument shows: Let us for now denote $\calA_\J^\infty(G)_d=\lim_{n\ra\infty}\calA^\infty_d(G)^{K_n,\J^n}$ and $\calA_\J^\infty(G)_{d+1}=\lim_{n\ra\infty}\calA^\infty_{d+1}(G)^{K_n,\J^n}$. By the very assertion of Prop.\ \ref{prop:same_d}, $\calA^\infty_d(G)^{K_n,\J^n}=\calA^\infty_{d+1}(G)^{K_n,\J^n}$ as sets for every $n\in \N$.  Estimating the seminorms in question, using only the fact that the adelic group norm $\|g\|$ is bounded away from $0$, cf.\ \eqref{eq:ungl}, one shows that the bijection given by the identity map $\calA^\infty_d(G)^{K_n,\J^n} \rightarrow \calA^\infty_{d+1}(G)^{K_n,\J^n}$ is continuous  for every $n\in \N$ and hence -- domain and target space being Fr\'echet -- a topological isomorphism by the Open Mapping Theorem. Here we note that the observation that $\calA^\infty_d(G)^{K_n,\J^n}=\calA^\infty_{d+1}(G)^{K_n,\J^n}$ as sets, which is all based on Prop.\ \ref{prop:same_d}, is crucial for the latter conclusion. Having said this, going over to the direct limits, $\calA_\J^\infty(G)_d=\calA_\J^\infty(G)_{d+1}$ also topologically. The reader may also find an analogue of this discussion (for real reductive groups and arithmetic sugroups) in \cite{li-sun}, \S3.1, there, however, without proofs.
\end{rem}

\subsection{The right regular action of $G(\A)$ on smooth-automorphic forms}\label{sect:smautrep}
The following result on $\calA_\J^\infty(G)$ is well-known and central for all our considerations later on. However, lacking a good reference and given certain topological issues and incompatibilities in the literature, cf.\ Remark-Warning \ref{rem:warning}, we prefer to give a full proof of it.

\begin{prop}\label{prop:smautorep}
Acted upon by right translation $R$, the LF-space $\calA_\J^\infty(G)$ becomes a smooth representation of $G(\A)$.
\end{prop}
\begin{proof}
We will first verify that the right regular action $R$ of $G(\A)$ defines a continuous map $G(\A)\times \calA_\J^\infty(G) \ra \calA_\J^\infty(G)$ and hence a representation of $G(\A)$: As $\calA_\J^\infty(G)$ is barrelled, it suffices to show that the latter map is separately continuous, cf.\ \S \ref{sect:rep}. 
To this end, let $g=(g_\infty,g_f)\in G(\A)$ arbitrary, but fixed, and consider the linear operator $R(g): \calA_\J^\infty(G) \ra \calA_\J^\infty(G)$. By construction of the LF-space topology on $\calA_\J^\infty(G)$ it suffices to show the continuity of the restrictions $R(g): \calA^\infty_d(G)^{K_n,\J^n} \ra \calA_\J^\infty(G)$ for each $n$, in order to obtain the desired continuity of $R(g)$, \S \ref{sect:LF}. So, let $n\in\N$ be arbitrary. Prop.\ \ref{prop:sm_inf} implies that $R((g_\infty,\id))=R_\infty(g_\infty)$ is a topological automorphism of the Fr\'echet space $\calA^\infty_d(G)^{K_n,\J^n}$. %and hence it is also continuous as a map 
%$$R((g_\infty,\id)): \calA^\infty_d(G)^{K_n,\J^n}\ra\calA_\J^\infty(G),$$
%as the subspace-topology of $\calA^\infty_d(G)^{K_n,\J^n}$ inside $\calA_\J^\infty(G)$ coincides with the original Fr\'echet topology, cf.\ \S \ref{sect:LF}. 
On the other hand, $R((\id,g_f))$ obviously defines a linear operator 
$R((\id,g_f)): \calA^\infty_d(G)^{K_n,\J^n} \ra \calA^\infty_d(G)^{g_f K_n g^{-1}_f,\J^n}$
whose image hence lies inside (any) Fr\'echet space $\calA^\infty_d(G)^{K_m,\J^m}, $ with $m\gg n$ such that $g_f K_n g^{-1}_f\supseteq K_m$ (such a $K_m$ exists, because $\{K_n\}_{n\in\N}$ defines a base of neighbourhoods of $\id$). Hence, continuity of 
$$R((\id,g_f)): \calA^\infty_d(G)^{K_n,\J^n} \ra \calA_\J^\infty(G)$$ 
follows from the fact that for every $X\in\U(\g)$ and $\varphi\in\calA_d^\infty(G)^{K_n,\J^n}$
\begin{align*}
	p_{d,X}(R((\id,g_f))\varphi)&\overset{\phantom{\eqref{eq:ungl}}}=\sup_{g\in G(\A)}|(X\varphi)(gg_f)|\norm g^{-d}\\
	&\overset{\eqref{eq:ungl}}\leq C_0^d\sup_{g\in G(\A)}|(X\varphi)(gg_f)|\norm{gg_f}^{-d}\norm{g_f}^{d}\\
	&\overset{\phantom{\eqref{eq:ungl}}}=C_0^d\norm{g_f}^d\,p_{d,X}(\varphi).
\end{align*}
In summary, $R(g)=R((\id,g_f))\circ R((g_\infty,\id))$ is a continuous map $\calA^\infty_d(G)^{K_n,\J^n} \ra \calA_\J^\infty(G)$ for each $n\in\N$, hence $R(g): \calA_\J^\infty(G) \ra \calA_\J^\infty(G)$ is continuous for all $g\in G(\A)$.

In order to prove continuity in the other variable, let $\varphi\in\calA_\J^\infty(G)$ be an arbitrary, but fixed smooth-automorphic form and consider its orbit map $c_\varphi$. Choose any $n\in\N$ such that $\varphi\in \calA^\infty_d(G)^{K_n,\J^n}$. Then, by Prop.\ \ref{prop:sm_inf}, $c_\varphi$ restricts to a smooth, and hence in particular continuous map, $c_\varphi: G_\infty \ra \calA^\infty_d(G)^{K_n,\J^n}$. As $\calA^\infty_d(G)^{K_n,\J^n}$ inherits from $\calA_\J^\infty(G)$ its original Fr\'echet topology, $c_\varphi$ also defines a continuous map $c_\varphi: G_\infty \ra \calA^\infty_\J(G)$. On the other hand, $\varphi$ being right-invariant under the open compact subgroup $K_n$, we also obtain a continuous restriction $c_\varphi: G(\A_f) \ra \calA^\infty_\J(G)$. Hence, by barrelledness of $\calA_\J^\infty(G)$ and the just shown continuity of $R((\id,g_f)): \calA^\infty_\J(G) \ra \calA_\J^\infty(G)$ the map $G(\A_f)\times \calA_\J^\infty(G) \ra \calA_\J^\infty(G)$, $(g_f,\phi)\mapsto R(g_f)\phi$ is jointly continuous. Therefore, $c_\varphi$, factoring as
\begin{align*}
		G(\A) 	& \stackrel{\sim}{\longrightarrow} 	& G_\infty \times G(\A_f) 	& \longrightarrow 	& \calA^\infty_\J(G) \times G(\A_f) 	& \longrightarrow 	& \calA^\infty_\J(G) &\\
		g 		& \longmapsto 						& (g_\infty, g_f) 			& \longmapsto 		& (c_\varphi(g_\infty), g_f) 		& \longmapsto		& R(g_f) c_\varphi(g_\infty)= & c_\varphi(g)
	\end{align*}
is continuous as it is a composition of continuous maps. 

Hence, $G(\A)\times \calA_\J^\infty(G) \ra \calA_\J^\infty(G)$, $(g,\varphi)\mapsto R(g)\varphi=\varphi(\rule{3mm}{0.15mm} \cdot g)$ is separately continuous and hence, by barrelledness of $\calA^\infty_\J(G)$, also jointly continuous. \\\\
We now prove smoothness. Combining Lem.\ \ref{prop:LFrepsm} and Prop.\ \ref{prop:sm_inf}, right translation on the LF-space $\calA_\J^\infty(G)$ is a smooth representation of $G_\infty$. To prove that it is also a smooth representation of $ G(\A_f) $, i.e., that topologically $\calA_\J^\infty(G)=\lim_{n\to\infty}\calA_\J^\infty(G)^{K_n}$, we need to argue that that the identity map
 $$\calA_\J^\infty(G)=\lim_{n\to\infty}\calA_d^\infty(G)^{K_n,\J^n} \overset{id}\longrightarrow \lim_{n\to\infty}\calA_\J^\infty(G)^{K_n}$$
is bicontinuous: By definition of the strict inductive limit topology, this is equivalent to that the restrictions
\begin{equation}\label{eq:incl2}
\calA_d^\infty(G)^{K_n,\J^n}\hookrightarrow\lim_{m\to\infty}\calA_\J^\infty(G)^{K_m} \quad\text{and}\quad \calA_\J^\infty(G)^{K_n}\hookrightarrow\calA_\J^\infty(G) 
\end{equation}
to the limit-steps are continuous for all $ n\in\N $. The second inclusion is continuous by construction, as is the map $\calA_d^\infty(G)^{K_n,\J^n}\hookrightarrow\calA_\J^\infty(G)$. However, as the image of the latter map lands inside $\calA_\J^\infty(G)^{K_n}\subset \calA_\J^\infty(G)$ and again, $\calA_\J^\infty(G)^{K_n}\hookrightarrow\lim_{m\to\infty}\calA_\J^\infty(G)^{K_m}$ is continuous by construction, also the first map in \eqref{eq:incl2} is continuous.
% ALTERNATIVE PROOF:
%We now prove smoothness. Combining Lem.\ \ref{prop:LFrepsm} and Prop.\ \ref{prop:sm_inf}, right translation on the LF-space $\calA_\J^\infty(G)$ is a smooth representation of $G_\infty$, i.e., $\calA_\J^\infty(G)^{\infty_\R}=\calA_\J^\infty(G)$ topologically. So, by the definition of the LCTVS of globally smooth vectors, cf.\ \ref{sect:rep}, we obtain identities of LCTVSs
%$$\calA_\J^\infty(G)^{\infty_\A}=\lim_{n\ra\infty} (\calA_\J^\infty(G)^{\infty_\R})^{K_n}=\lim_{n\ra\infty} \calA_\J^\infty(G)^{K_n}=\lim_{n\ra\infty} \left(\lim_{n\leq k\ra\infty}\calA_\J^\infty(G)^{K_n,\J^k}\right)$$
%where the last equality follows from the openness of the projections $E^{K_n}$ and Lem.\ \ref{lem:useful-lemma}. However, by Grothendieck's factorization theorem, \cite{grothendieck_topovs}, Chp.\ 4, Part 1, \S 3, Prop.\ 3, the bounded sets of the double-limit $\lim_{n\ra\infty} \left(\lim_{n\leq k\ra\infty}\calA_\J^\infty(G)^{K_n,\J^k}\right)$ and of $\calA_\J^\infty(G)=\lim_{n\ra\infty}\calA_\J^\infty(G)^{K_n,\J^n}$ coincide, whence the identity map between the latter two spaces is bi-bounded, hence bi-continious by \cite{grothendieck_topovs}, Chp.\ 4, Part 1, \S 1, Cor.\ 1.
\end{proof}

\begin{rem}\label{rem:phicont}
We recall that any smooth function $f:G(\A)\ra\C$ is automatically continuous. In fact, for smooth-automorphic forms this also follows from the the fact, shown in the proof of Prop.\ \ref{prop:smautorep}, that the orbit maps $c_\varphi: G(\A)\ra \calA_\J^\infty(G)$ are all continuous, as $\varphi = ev_\id\circ c_\varphi$.
%that  each $\varphi\in\calA_\J^\infty(G)$ is continuous. Indeed, it follows directly from the definition of the seminorms $p_{d,X}$ on the steps $\calA_d^\infty(G)^{K_n,\J^n}$ that the evaluation-map at the identity, $ev_\id: \calA_d^\infty(G)^{K_n,\J^n}\ra\C$, is a bounded linear functional, hence continuous, as the spaces $\calA_d^\infty(G)^{K_n,\J^n}$, being Fr\'echet, are all bornological. Therefore, $ev_\id: \calA_\J^\infty(G)\ra\C$ is continuous and hence so is $\varphi = ev_\id\circ c_\varphi$ for all $\varphi\in\calA_\J^\infty(G)$.
\end{rem}

\subsection{(Relation to) $ K_\infty $-finite automorphic forms}

Obviously, the space of (usual) automorphic forms $\calA_\J(G)$, cf.\ \cite{bojac}, \S 4.2, which are annihilated by a power of the fixed ideal $\J$, identifies as the dense subspace of right-$K_\infty$-finite vectors in $\calA_\J^\infty(G)$\footnote{Reading the definition in \cite{bojac}, \S 4.2 very carefully, in order to render the identity $\calA_\J(G)=\calA_\J^\infty(G)_{(K_\infty)}$ really obvious, one should recall Rem.\ \ref{rem:phicont} from above.}. It is very well known that $\calA_\J(G)$ is a $(\g_\infty,K_\infty,G(\A_f))$-module (but it also follows directly from Prop.\ \ref{prop:smautorep} above, as $\calA_\J(G)=\calA_\J^\infty(G)_{(K_\infty)}=\calA_\J^\infty(G)^{\infty_\A}_{(K_\infty)}$). Although dense in $\calA^\infty_\J(G)$, and hence, topologically ``almost all'' of $\calA^\infty_\J(G)$, the space $\calA_\J(G)$ is much smaller than $\calA^\infty_\J(G)$ from the perspective of vector spaces: Indeed, $\calA_\J(G)$ is of countable dimension by a theorem of Harish-Chandra, cf.\ \cite{hch}, Thm.\ 1 (see also \cite{bojac}, Thm.\ 1.7 and \S 4.3.(i) therein), whereas  $\calA^\infty_\J(G)$ is of uncountable dimension (since it contains infinite-dimensional Fr\'echet spaces $ \calA^\infty_d(G)^{K_n,\J^n} $).

\section{Smooth-automorphic representations}\label{sect:smautreps}

\subsection{Smooth-automorphic representations}
In this paper we propagate the idea to give preference to $\calA^\infty_\J(G)$, rather than to $\calA_\J(G)$, as the former allows a representation of $G(\A)$ and not only the structure of a $(\g_\infty,K_\infty,G(\A_f))$-module (which breaks the symmetry between the role of the archimedean and the non-archimedean factors of $G(\A)$). Moreover, to round this up by a subtlety (which may suit the taste of purists among the readers), the space of functions $\calA_\J(G)$ changes with the very choice of $K_\infty$, whereas $\calA^\infty_\J(G)$ is independent of any particular additional choices: Once $\J$ is given, $\calA^\infty_\J(G)$ depends on nothing else than the group-scheme $G/F$ itself (which we believe is a much more satisfactory setup). \\\\
In order to substantiate this approach, we will show here that the representation theory evolving out of $\calA^\infty_\J(G)$ is rich enough in order to recover the representation-theoretical phenomena in automorphic forms.\\\\ 
Indeed, as a first step and as some sort of ground-work, in this section we shall prove an ``automorphic analogue'' of a famous result of Harish-Chandra (on admissible $G_\infty$-representations and their underlying $(\g_\infty,K_\infty)$-modules); moreover, we will establish a natural 1-to-1 correspondence between the irreducible smooth-automorphic representations of $G(\A)$ (i.e., irreducible $G(\A)$-subquotients of $\calA^\infty_\J(G)$) and the usual irreducible automorphic representations (i.e., irreducible $(\g_\infty,K_\infty,G(\A_f))$-module subquotients of $\calA_\J(G)$); finally, we will also verify the fundamental and all-important local-global property of irreducible smooth-automorphic representations, provided by a topologized version of the restricted tensor product theorem.\\\\
In analogy with the classical definition of automorphic representations, we introduce the following

\begin{defn}\label{defn:001}
	A $ G(\A) $-representation $ (\pi,V) $ is a \textit{smooth-automorphic representation} if it is equivalent to a quotient $ U/W $, where $ W\subseteq U $ are $G(\A)$-subrepresentations of $ \calA^\infty_\J(G) $ (for some ideal $\J$ of $\calZ(\g)$ of finite codimension). Moreover, if $ W=0 $, we say that $ (\pi,V) $ is a \textit{smooth-automorphic subrepresentation}.
\end{defn}
We shortly observe that by Lem.\ \ref{lem:subsquots} and Prop.\ \ref{prop:smautorep} the word ``smooth'' in our terminus ``smooth-automorphic representation'' does not amount to an abuse of terminology as every smooth-automorphic representation is indeed a smooth $ G(\A) $-representation in the sense of \S\ref{sect:G(A)reps}.

\begin{rem}\label{rem:001}
	By definition, the notion of a smooth-automorphic representation entails the assumption that the quotient $U/W$ is complete, as it is supposed to be a representation of $ G(\A) $. We warn the reader that in general, the quotient of a complete LCTVS by a closed subspace does not need to be complete, nor is it automatic that the quotient of two $G(\A)$-representations is again a representation, cf.\ Rem.\ \ref{rmk:qtreps}. It is our conjecture, however, that for all $ G(\A) $-subrepresentations $ W\subseteq U \subseteq \calA^\infty_\J(G) $, the quotient $ U/W $ is complete and defines a smooth-automorphic representation. Let us point out that according to Lem.\ \ref{lem:subsquots}, this is certainly the case, if each quotient $U^{K_n}/W^{K_n}$ is complete and barrelled. It follows that $U/W$ is in particular complete, if there exists an $m\in\N$ such that $\J^m U=0$, as then $U^{K_n}$ is a closed subspace of the Fr\'echet space $\calA^\infty_d(G)^{K_n,\J^n}$ for all $n\geq m$ and hence Fr\'echet itself. Since the actions of $\mathcal Z(\g)$ and $G(\A)$ commute, the latter applies in particular, if $U$ is finitely generated as a $G(\A)$-representation. 
\end{rem}

The following lemma gives a first glimpse into the relationship between the classical automorphic representations and smooth-automorphic representations.

\begin{lem}\label{lem:008}
	Let $ V_0 $ be a $ (\g_\infty,K_\infty,G(\A_f)) $-submodule of $ \calA_\J(G) $. Then, the topological closure of $V_0$ in $\calA_\J^\infty(G)$, $ V:=\Cl_{\calA_\J^\infty(G)}(V_0) $, is a smooth-automorphic subrepresentation.
\end{lem}

\begin{proof}
	We need to prove that $ V $ is $ G(\A) $-invariant. Let $ U $ be the smallest closed $ G(\A) $-invariant subspace of $ \calA_\J^\infty(G) $ containing $ V_0 $:
	\begin{equation}\label{eq:007}
	U=\Cl_{\calA_\J^\infty(G)}(\mathrm{span}_\C R(G(\A))V_0).
	\end{equation}
	Obviously, $ U\supseteq V $. So, suppose that $ U\supsetneqq V $. Then, by the Hahn-Banach theorem there exists a non-zero continuous linear functional $ b:U\to\C $ such that $ b\big|_V=0 $. For every $\phi\in V_0 $, let us look at the function $ \varphi_\phi\in C^\infty(G(\A)) $,
	\[ \varphi_\phi(g):=b(R(g)\phi),\qquad g\in G(\A). \]
	For every $ X\in\g_\infty $, we have
	\[ (X\varphi_\phi)(g)=\frac d{dt}\Big|_{t=0}b(R(g\exp(tX))\phi)=b(R(g)R(X)\phi)=\varphi_{X\phi}(g),\qquad g\in G(\A). \]
	It follows that
	\begin{equation}\label{eq:004}
	(X\varphi_\phi)(g)=\varphi_{X\phi}(g)=b(R(g)R(X)\phi),\qquad X\in\U(\g), \ g\in G(\A). 
	\end{equation}
	In particular, for $ n\in\N $ such that $\J^n \phi=0 $, we have $ \J^n\varphi_\phi=0 $. Note also that $ \varphi_\phi $ is $ K_\A $-finite on the right. Denoting by $ G_\infty^\circ $ the identity component of $ G_\infty $, it follows that for every $g_f\in G(\A_f) $, the function $ \varphi_\phi(\rule{3mm}{0.15mm} \cdot g_f): G_\infty^\circ\to\C $ is smooth, $ K_\infty $-finite on the right and $ \calZ(\g) $-finite, hence it is real analytic (see, e.g., \cite{borel}, 3.15). Moreover, we have
	\[ X(\varphi_\phi(\rule{3mm}{0.15mm} \cdot g_f))(\id_{G_\infty})=(X\varphi_\phi)(g_f)\overset{\eqref{eq:004}}=b(\underbrace{R(g_f)R(X)\phi}_{\in V})=0,\qquad X\in\U(\g), \]
	hence $\varphi_\phi(g_\infty g_f)=0$, for all $g_\infty\in G_\infty^\circ$ and $g_f\in G(\A_f)$.
	This means that $ b $ vanishes on
	\begin{equation}\label{eq:006}
	\begin{aligned}
	\Cl_{\calA_\J^\infty(G)}\left(\mathrm{span}_\C R(G_\infty^\circ\times G(\A_f))V_0\right)
	&=\Cl_{\calA_\J^\infty(G)}\left(\mathrm{span}_\C R(G_\infty^\circ\times G(\A_f))R(K_\infty)V_0\right)\\
	&=\Cl_{\calA_\J^\infty(G)}\left(\mathrm{span}_\C R(G(\A))V_0\right)
	\overset{\eqref{eq:007}}=U,
	\end{aligned}
	\end{equation}
	where the second equality holds because $ K_\infty $ meets every connected component of $ G_\infty $.	Thus, $ b $ is identically zero, which is a contradiction. Hence, $V=U$ and so $V$ is $G(\A)$-stable.
\end{proof}

\subsection{The general dictionary I: Admissibility and an automorphic variant of a theorem of Harish-Chandra}

The following result, see Thm.\ \ref{thm:005} below, provides a fundamental dictionary between smooth-automorphic subrepresentations of $\calA_\J^\infty(G)$ and automorphic subrepresentations of $ \calA_\J(G)$ (where the expression ``automorphic subrepresentation'' is used in the usual way, i.e., denoting a $ (\g_\infty,K_\infty,G(\A_f)) $-submodule): We show that the irreducible (and, more generally, even all the admissible) $G(\A)$-subrepresentations of $ \calA_\J^\infty(G) $ are {\it exactly} the topological closures of the irreducible (resp., admissible) automorphic subrepresentations of $ \calA_\J(G) $ within $ \calA_\J^\infty(G) $. We invite the reader to view this result of ours as a {\it global}, or {\it automorphic} analogue of Harish-Chandra's result, providing a 1:1-correspondence between the $G_\infty$-subrepresentations of a given admissible $G_\infty$-representation $V$ and the $(\g_\infty,K_\infty)$-submodules of its underlying $(\g_\infty,K_\infty)$-module, cf.\ \cite{hch53} or \cite{varada}, Thm.\ II.7.14.

\begin{thm}\label{thm:005}
The admissible smooth-automorphic subrepresentations stand in one-one correspondence with the admissible $ (\g_\infty,K_\infty,G(\A_f)) $-submodules $ V_0 $ of $ \calA_\J(G) $, the correspondence $ V\leftrightarrow V_0 $ being
	\[ V_0=V_{(K_\infty)}\qquad\text{and}\qquad V=\Cl_{\calA^\infty_\J(G)}(V_0). \]
Moreover, the above correspondence respects irreducibility. Hence, every admissible (resp., irreducible) automorphic subrepresentation of $\calA_\J(G)$ lifts to an admissible (resp., irreducible) smooth-automorphic subrepresentation, recovering the original automorphic representation as its space of $K_\infty$-finite vectors. 
\end{thm}

\begin{proof} Firstly, we recall that by \cite{bojac}, Prop.\ 4.5.(4) every irreducible automorphic subrepresentation $V_0$ is automatically admissible. Hence, if we manage to show that the underlying $ (\g_\infty,K_\infty,G(\A_f)) $-module of an irreducible smooth-automorphic subrepresentation $V$ of $\calA^\infty_\J(G)$ remains irreducible, then -- $V$ being a smooth $G(\A)$-representation by Lem.\ \ref{lem:subsquots} and Prop.\ \ref{prop:smautorep}, as already observed above -- $V$ is admissible, too. With this observation in mind, it therefore suffices to prove the following four claims:
	\begin{enumerate}
		\item\label{enum:001:1} Let $ V_0 $ be an admissible $ (\g_\infty,K_\infty,G(\A_f)) $-submodule of $ \calA_\J(G) $. Then, $ V:=\Cl_{\calA^\infty_\J(G)}(V_0) $ is an admissible $ G(\A) $-representation, and we have $ V_{(K_\infty)}=V_0 $.
		\item\label{enum:001:2} Let $ V $ be an admissible smooth-automorphic subrepresentation. Then, $ V_0:=V_{(K_\infty)} $ is an admissible  $ (\g_\infty,K_\infty,G(\A_f)) $-module, and we have $ \Cl_{\calA^\infty_\J(G)}(V_0)=V $.		
		\item\label{enum:001:3} Let $ V_0 $ be an irreducible $ (\g_\infty,K_\infty,G(\A_f)) $-submodule of $ \calA_\J(G) $. Then, $ V:=\Cl_{\calA^\infty_\J(G)}(V_0) $ is an irreducible $ G(\A) $-representation.
		\item\label{enum:001:4} Let $ V $ be an irreducible smooth-automorphic subrepresentation. Then, the $ (\g_\infty,K_\infty,G(\A_f)) $-module $ V_{(K_\infty)} $ is irreducible.
	\end{enumerate}

Ad \eqref{enum:001:1}: By Lem.\ \ref{lem:008}, $ V $ is a $ G(\A) $-subrepresentation of $\calA^\infty_\J(G)$, and hence, smooth by Lem.\ \ref{lem:subsquots} and Prop.\ \ref{prop:smautorep}. It is obvious from the construction that $V_0$ is a $ (\g_\infty,K_\infty,G(\A_f)) $-submodule of $V_{(K_\infty)}$ and dense in $V$. As $V_0$ is furthermore assumed to be admissible, Lem.\ \ref{lem:dense_subm} implies that $ V_{(K_\infty)}=V_0 $ and that $V$ is admissible as a $G(\A)$-representation.

Ad \eqref{enum:001:2}: The first part of the claim holds by definition, and the second one by \cite{hch66}, Lem.\ 4.  
	
Ad \eqref{enum:001:3}: As we observed above, $V_0$ is admissible, so by \eqref{enum:001:1}, $ V $ is a $ G(\A) $-subrepresentation of $\calA^\infty_\J(G)$ and $ V_{(K_\infty)}=V_0 $. Let $ U\neq0 $ be a closed $ G(\A) $-invariant subspace of $ V $. Since $ U_{(K_\infty)} $ is dense in $ U $ (again, see \cite{hch66}, Lem.\ 4), $ U_{(K_\infty)}\neq0 $, which by irreducibility of $ V_0 $ implies that $ U_{(K_\infty)}=V_0 $, hence $ U=\Cl_{V}(V_0)=V $.
	
Ad \eqref{enum:001:4}: Since $ V_{(K_\infty)} $ is dense in $ V $ (\cite{hch66}, Lem.\ 4), $ V_{(K_\infty)}\neq0 $. Thus, it suffices to prove that for every $ \phi\in V_{(K_\infty)}\setminus\{0\} $, we have
	\[ V_{0,\phi}:=\left<\phi\right>_{(\g_\infty,K_\infty,G(\A_f))}=V_{(K_\infty)}. \]
Combining Lemma \ref{lem:008} and the irreducibility of $V$, we get $\Cl_{\calA^\infty_\J(G)} (V_{0,\phi}) =V$. Thus, by applying \eqref{enum:001:1} to $ V_{0,\phi} $ (which, being spanned by one single automorphic form, is admissible by \cite{bojac}, Prop.\ 4.5.(4)), we get that $ V_{(K_\infty)}=V_{0,\phi} $.
\end{proof}

%To prove a first consequence of Thm.\ \ref{thm:005}, we recall that f

%\begin{cor}\label{cor:021}
%	Let $ \left\{V_i\right\}_{i\in I} $ be a family of admissible $ (\g_\infty,K_\infty,G(\A_f)) $-submodules of $ \calA_\J(G) $ whose sum is direct. Then, the sum of automorphic subrepresentations $ \overline{V_i}:=\Cl_{\calA^\infty_\J(G)}(V_i) $ is also direct.
%\end{cor}

%\begin{proof}
%	Suppose that $ \sum_{i\in I}f_i=0 $ for some $ f_i\in\overline{V_i} $, where $ f_i=0 $ for all but finitely many $ i $. Then, for every $ \rho\in\widehat{K_\infty} $, $ \sum_{i\in I}E_\rho f_i=0 $. Since $ E_\rho f_i\in \left(\overline{V_i}\right)_{(K_\infty)}=V_i $ by Thm.\ \ref{thm:005}, and the sum of $ V_i $'s is direct, it follows that $ E_\rho f_i=0 $ for all $ \rho $ and $ i $. Thus, since $ \calA_\J^\infty(G) $ is a smooth representation of $ K_\infty $, by \cite[Lem.\ 5]{hch66} we have that $ f_i=\sum_{\rho\in\widehat{K_\infty}}E_\rho f_i=0 $ for every $ i\in I $.
%\end{proof}

\subsection{The general dictionary II: Extension to all irreducibles and smooth-automorphic Casselman-Wallach representations}
Thm.\ \ref{thm:005} provides (in particular) a dictionary between the irreducible smooth-automorphic subrepresentations of $\calA_\J^\infty(G)$ and the irreducible automorphic subrepresentations of $ \calA_\J(G)$. For a completely general understanding of the internal representation theory of the space of smooth-automorphic forms it is essential, however, to extend this comparison from irreducible subrepresentations to all irreducible subquotients, as only those will capture the representation-theoretical phenomena of  $\calA_\J^\infty(G)$ in sufficient generality. This section is devoted to such a general comparison of irreducibles.\\\\
We begin our analysis by studying smooth-automorphic Casselman-Wallach representations of $ G(\A) $ (see \S\ref{subsec:cass_wall}), examples of which we describe in Prop.\ \ref{prop:smmf} and its fundamental corollary, Cor.\ \ref{cor:fg} (representations, spanned by one smooth-automorphic form $\varphi$), below:

\begin{prop}\label{prop:smmf}
	Let $(\pi,V)$ be a smooth-automorphic subrepresentation that is annihilated by a power of $ \J $ (e.g., an irreducible smooth-automorphic subrepresentation). Then, $V$ is a Casselman-Wallach representation of $ G(\A) $. In particular, $V$ is an LF-space.

\end{prop}

\begin{proof}
	Let $ k\in\N $ such that $ \J^kV=0 $. For every $ n\in\N $, $V^{K_n}$ is a $G_\infty$-subrepresentation of $\calA^\infty_d(G)^{K_m,\J^m}$, where $ m:=\max(k,n) $. Hence, combining Prop.\ \ref{prop:sm_inf} with \cite{wallachII}, Lem.\ 11.5.2,  $V^{K_n}$ inherits from $\calA^\infty_d(G)^{K_m,\J^m}$ the structure of a Casselman-Wallach representation of $ G_\infty $. As $V$ is furthermore a smooth $G(\A)$-representation, cf.\ Lem.\ \ref{lem:subsquots} and and Prop.\ \ref{prop:smautorep}, $V$ is a Casselman-Wallach representation of $ G(\A) $ by definition and hence $V$ is an LF-space. 
	
	To prove that every irreducible smooth-automorphic subrepresentation $ (\pi,V) $ satisfies the assumption of the proposition, recall that by the irreducibility of $ V $ and the continuity of the right regular action $R$ we have $ V=\Cl_{\calA_\J^\infty(G)}(\mathrm{span}_\C R(G(\A))\varphi) $ for any non-zero $\varphi\in V$. Hence, by the continuity of the action of $\U(\g)$, $ \J^kV=0 $ for every $ k\in\N $ such that $ \J^k\varphi=0 $.
\end{proof}

We shall need the following variant of \eqref{eq:Erhogl}: For a $ G(\A) $-representation $(\pi,V)$ and an irreducible $K_\infty$-representation $(\rho_\infty,W_\infty)$, the usual continuous projection $E_{{\rho_\infty},V}: V\ra V$ onto the $\rho_\infty$-isotypic component of $V$ is defined by
\begin{equation}\label{eq:Erhoinf}
E_{\rho_\infty,V}(v):= d(\rho_\infty) \ \int_{K_\infty} \overline{\xi_{\rho_\infty}(k_\infty)} \ \pi(k_\infty)v \ dk_\infty.
\end{equation}
When the ambient representation $V$ is clear from the context, we will write $E_{\rho_\infty}=E_{\rho_\infty,V}$. The following corollary of Prop.\ \ref{prop:smmf} is fundamental:

\begin{cor}\label{cor:fg}
	Let $ \varphi\in\calA_\J^\infty(G) $. Then, the $ (\g_\infty,K_\infty,G(\A_f)) $-module
	\[ V_{0,\varphi}:=\sum_{\rho_\infty}\left<E_{\rho_\infty}(\varphi)\right>_{(\g_\infty,K_\infty,G(\A_f))}, \]
	spanned by the $(\g_\infty,K_\infty,G(\A_f))$-modules generated by the images $E_{\rho_\infty}(\varphi)$, $ \rho_\infty$ ranging through the equivalence-classes of irreducible $K_\infty$-representations, is annihilated by a power of $\J$ and is an admissible and finitely generated $ (\g_\infty,K_\infty,G(\A_f)) $-submodule of $ \calA_\J(G) $. Moreover, the $G(\A)$-subrepresentation of $ \calA_\J^\infty(G) $ generated by $\varphi$,
	\[ V_\varphi:=\Cl_{\calA_\J^\infty(G)}(\mathrm{span}_\C R(G(\A))\varphi) = \Cl_{\calA_\J^\infty(G)}(V_{0,\varphi}),\]
	is a Casselman-Wallach representation of $ G(\A) $. In particular, $V_\varphi$ is an LF-space and the right regular action $R$ of $G(\A)$ on $V_\varphi$ defines an admissible $G(\A)$-representation.
\end{cor}
\begin{proof}
	Let us fix $ m\in\N $ such that $ \varphi\in\calA_{d}^\infty(G)^{K_m,\J^m} $. By construction, $V_{0,\varphi}\subseteq\calA_\J(G)$ is annihilated by $\J^m$. Thus, for every irreducible representation $ \rho_\infty$ of ${K_\infty} $ and $ n\in\N $, we have
	\[ E_{\rho_\infty}^{K_n} (V_{0,\varphi}) \subseteq\left\{\phi\in\calA_\J(G):\J^m\phi=0\text{ and }E_{\rho_\infty}^{K_n}(\phi)=\phi\right\}. \]
	As the space on the right-hand side is finite-dimensional, cf.\ \cite{bojac}, \S 4.3(i), $V_{0,\varphi}$ is an admissible $ (\g_\infty,K_\infty,G(\A_f)) $-module. Consequently, $ V_{0,\varphi}^{K_m} $ is an admissible and $ \calZ(\g) $-finite $ (\g_\infty,K_\infty) $-module, hence  a finitely generated $ (\g_\infty,K_\infty) $-module, cf.\ \cite{vogan}, Cor.\ 5.4.16. In order to deduce that hence $V_{0,\varphi}$ is finitely generated, observe that it is generated as a $ G(\A_f) $-module by $\sum_{\rho_\infty}\left<E_{\rho_\infty}(\varphi)\right>_{(\g_\infty,K_\infty)}$, which, by construction, is a $ (\g_\infty,K_\infty) $-submodule of $ V_{0,\varphi}^{K_m} $. 	This shows the first assertion. For the second assertion, observe that $V_{0,\varphi}$ being an admissible $ (\g_\infty,K_\infty,G(\A_f)) $-submodule of $ \calA_\J(G) $, by our Thm.\ \ref{thm:005}, $\Cl_{\calA_\J^\infty(G)}(V_{0,\varphi})$ is an admissible subrepresentation of $ \calA_\J^\infty(G) $, and it is obviously annihilated by $ \J^m $, so it is a Casselman-Wallach representation of $ G(\A) $ by Prop.\ \ref{prop:smmf}. It remains to prove that $V_\varphi=\Cl_{\calA_\J^\infty(G)}(V_{0,\varphi})$. The inclusion $\Cl_{\calA_\J^\infty(G)}(V_{0,\varphi}) \subseteq V_\varphi $ is obvious. On the other hand, by \cite[Lemma 5]{hch66} we have that $ \varphi=\sum_{\rho_\infty}E_{\rho_\infty}(\varphi)\in \Cl_{\calA_\J^\infty(G)}V_{0,\varphi} $, hence $ V_\varphi\subseteq \Cl_{\calA_\J^\infty(G)}(V_{0,\varphi}) $.
\end{proof}

\begin{rem}
We remark again that the sheer fact that $V_\varphi$ is a closed subspace of an LF-space, namely $\calA_\J^\infty(G) $, does not imply that $V_\varphi$ is necessarily an LF-space itself, whence this needed an extra argument. Moreover, it does not follow from purely abstract considerations that a $G(\A)$-representation, spanned by one single function, is necessarily admissible (and hence even less so necessarily a Casselman-Wallach representation): For instance, this will even fail for (equivalence classes of) functions $f\in L^2_{dis}(G(F)A^\R_G\backslash G(\A))$ in the discrete part of the $L^2$-spectrum. Indeed, the reader may easily construct a function $f\in L^2_{dis}(G(F)A^\R_G\backslash G(\A))$ such that the corresponding $G(\A)$-subrepresentation $$V_f:=\Cl_{L^2_{dis}(G(F)A^\R_G\backslash G(\A))}(\mathrm{span}_\C R(G(\A))f)$$ of the Hilbert-space $L^2_{dis}(G(F)A^\R_G\backslash G(\A))$ is not admissible. We remark that the $\mathcal Z(\g)$-finiteness of $\varphi\in\calA_\J^\infty(G)$ was essential in order to obtain Cor.\ \ref{cor:fg}
\end{rem}
We are now ready to prove a general result, comparing irreducible smooth-automorphic representations with irreducible automorphic representations: 

\begin{thm}\label{prop:subquots}
	If $ V $ is an irreducible smooth-automorphic representation, then $ V_{(K_\infty)} $ is an irreducible automorphic representation. In particular, every irreducible smooth-automorphic representation is admissible. Conversely, let $ V_0 $ be an irreducible automorphic representation of $ G(\A) $. Then, there exists an irreducible smooth-automorphic representation $V $ such that $ V_{(K_\infty)}\cong V_0 $.
\end{thm}

\begin{proof}
	Let $ W\subseteq U $ be closed $ G(\A) $-invariant subspaces of $ \calA^\infty_\J(G) $ such that $ V:=U/W $ is an irreducible $ G(\A) $-representation. We show that the $ (\g_\infty,K_\infty,G(\A_f)) $-module $ V_{(K_\infty)} $ is irreducible. Let us denote by $ \underline \varphi $ (resp., $\underline M$) the image of an arbitrary $ \varphi\in\calA^\infty_\J(G) $ (resp., $ M\subseteq\calA^\infty_\J(G) $) under the canonical epimorphism $ \calA^\infty_\J(G)\twoheadrightarrow\calA^\infty_\J(G)/W $.
	Since $ U_{(K_\infty)} $ is dense in $ U $ (\cite{hch66}, Lem.\ 4), we have $ U_{(K_\infty)}\setminus W\neq\emptyset $. Thus, it suffices to prove that for every $ \phi\in U_{(K_\infty)}$, which is not in $W $, the $ (\g_\infty,K_\infty,G(\A_f)) $-submodule
	\[ 	V_{0,\phi}:=\left<\underline \phi\right>_{(\g_\infty,K_\infty,G(\A_f))}=\underline{\left<\phi\right>_{(\g_\infty,K_\infty,G(\A_f))}} \]
	of $V_{(K_\infty)}$ equals $ V_{(K_\infty)} $. To this end, note that by Lem.\ \ref{lem:008},
	\[ \Cl_U\left(\left< \phi \right>_{(\g_\infty,K_\infty,G(\A_f))}+W_{(K_\infty)}\right) \]
	is a closed $ G(\A) $-invariant subspace of $ U $. As it obviously contains $ W $ as a proper subspace, by the irreducibility of $ V $ it equals $U$. In particular, $ U_{0,\phi}:=\left< \phi \right>_{(\g_\infty,K_\infty,G(\A_f))}+W_{(K_\infty)} $ is dense in $ U $, which implies that $ \underline{U_{0,\phi}}=V_{0,\phi} $ is dense in $ \underline U=V $. As $V_{0,\phi}$ is moreover an admissible $(\g_\infty,K_\infty,G(\A_f))$-submodule of $V_{(K_\infty)}$ by \cite{bojac}, Prop.\ 4.5.(4), and as $V$ is a smooth $G(\A)$-representation by Lem.\ \ref{lem:subsquots} and and Prop.\ \ref{prop:smautorep}, Lem.\ \ref{lem:dense_subm} finally implies that $ V_{0,\phi}=V_{(K_\infty)} $, as desired. Hence,  $ V_{(K_\infty)} $ is irreducible. At the same time, the other implication of Lem.\ \ref{lem:dense_subm} shows that $V$ is admissible as claimed.\\\\
	We will now prove the second assertion: To prove the existence of $V$, let $ W_0\subseteq U_0 $ be $ (\g_\infty,K_\infty,G(\A_f)) $-submodules of $ \calA_\J(G) $ such that $ U_0/W_0\cong V_0 $. Let us fix a $ \phi_0\in U_0$, which is not in $W_0 $ and denote $ U_1:=\left<\phi_0\right>_{(\g_\infty,K_\infty,G(\A_f))}\subseteq U_0 $. By the irreducibility of $ U_0/W_0 $, the canonical homomorphism $ U_1\to U_0/W_0 $ is surjective, hence, denoting its kernel by $ W_1 $, we have the following isomorphisms of $ (\g_\infty,K_\infty,G(\A_f)) $-modules: 
	\begin{equation}\label{eq:012}
	U_1/W_1\cong U_0/W_0\cong V_0.
	\end{equation}
	Let us fix $ k\in\N $ such that $ \J^k\phi_0=0 $. Since by \cite{bojac}, Prop.\ 4.5.(4) the $ (\g_\infty,K_\infty,G(\A_f)) $-submodules $ W_1\subseteq U_1 $ of $ \calA_\J(G) $ are admissible, Thm.\ \ref{thm:005} implies that their closures $ \overline W_1:=\Cl_{\calA_\J^\infty(G)}(W_1) $ and $ \overline U_1:=\Cl_{\calA_\J^\infty(G)}(U_1) $ are admissible smooth-automorphic subrepresentations satisfying
	\begin{equation}\label{eq:011}
	(\overline W_1)_{(K_\infty)}=W_1\qquad\text{and}\qquad (\overline U_1)_{(K_\infty)}=U_1.
	\end{equation}
	Since moreover the spaces $ \overline W_1\subseteq\overline U_1 $ are obviously annihilated by $ \J^k $, by Prop.\ \ref{prop:smmf} they are Casselman-Wallach representations of $ G(\A) $, and hence so is their quotient $ V:=\overline U_1/\overline W_1 $ by Lem.\ \ref{lem:cass_wall_reps}.\eqref{lem:cass_wall_reps:1}. Moreover, we have the following isomorphisms of $ (\g_\infty,K_\infty,G(\A_f)) $-modules:
	\[ \begin{aligned}
	V_{(K_\infty)}&=\bigoplus_{\rho\in\widehat{K_\infty}}E_\rho(\overline U_1/\overline W_1)
	=\bigoplus_{\rho\in\widehat{K_\infty}}(E_\rho\overline U_1+\overline W_1)/\overline W_1
	=((\overline U_1)_{(K_\infty)}+\overline W_1)/\overline W_1\\
	&\overset{\eqref{eq:011}}=(U_1+\overline W_1)/\overline W_1
	\cong U_1/W_1
	\overset{\eqref{eq:012}}\cong V_0,
	\end{aligned} \]
	where the inverse of the next-to-last isomorphism is just the canonical map that is obviously surjective and is injective because $ U_1\cap\overline W_1=(\overline W_1)_{(K_\infty)}\overset{\eqref{eq:011}}=W_1 $. Finally, the irreducibility of $ V_{(K_\infty)}\cong V_0 $ implies the irreducibility of $ G(\A) $-representation $ V $ as in the proof of claim \eqref{enum:001:3} in the proof of Thm.\ \ref{thm:005}.
\end{proof}

\subsection{The general dictionary III: The local-global principle and the restricted tensor product theorem}

As it is implicit in the proof of Thm.\ \ref{prop:subquots}, the underlying $ (\g_\infty,K_\infty,G(\A_f)) $-module $ V_{(K_\infty)} $ of every irreducible smooth-automorphic representation $V$ allows a completion as a Casselman-Wallach representation of $G(\A)$. By Lem.\ \ref{lem:cass_wall_reps}.\eqref{lem:cass_wall_reps:2} this completion is in fact unique up to isomorphism of $G(\A)$-representations. \\\\
Our next result provides the necessary local-global principle for all such irreducible smooth-automorphic Casselman-Wallach representations. Its underlying algebraic assertion seems well-known to experts (see \cite{cogdell_fields}, Thm.\ 3.4, for $G=\GL_n$), albeit the decisive question, which is the correct choice of a locally convex topology on restricted tensor products, seems to remain open in the available literature, whence, lacking a precise reference, we decided to include the result in our paper and fill this gap. %(including the unpublished \cite{moore}). 
In order to state the result, we let $\cprojtp$, resp.\ $\cindtp$, denote the completed projective tensor product, resp.\ completed inductive tensor product, of LCTVSs. We refer to \cite{grothendieck_book}, I, \S 1, n$^\circ$'s 1--3 and \cite{warner}, App.\ 2.2, for their basic properties.

\begin{thm}[Tensor product theorem]\label{thm:TPthm}
	Let $(\pi,V)$ be an irreducible smooth-automorphic Casselman-Wallach representation of $G(\A)$. Then, for each $\vv\in S$, there is an irreducible smooth admissible representation $(\pi_\vv,V_\vv)$ of $G(F_\vv)$, which is furthermore of moderate growth, if $\vv\in S_\infty$, such that as $G(\A)$-representations
	\begin{equation}\label{eq:TPthm}
	\pi \cong \bigcprojtp_{\vv\in S_\infty} \pi_\vv \ \cindtp \ \rtprod_{\vv\notin S_\infty}\pi_\vv,
	\end{equation}
	where the restricted tensor product $\rtprod_{\vv\notin S_\infty}\pi_\vv$ is endowed with the finest locally convex topology. Among all representations with the aforementioned properties, the $\pi_\vv$ are unique up to isomorphy. 
\end{thm}
\begin{proof}
	By Thm.\ \ref{prop:subquots}, $V_{(K_\infty)}$ is an irreducible admissible $(\g_\infty,K_\infty,G(\A_f))$-module. It hence follows from the classical restricted tensor product theorem, \cite{flath}, Thm.\ 3, and our \S \ref{sect:genreps}, that there are irreducible admissible $(\g_\vv ,K_\vv )$-modules $(\pi_{0,\vv },V_{0,\vv })$, $\vv \in S_\infty$, and irreducible smooth admissible $G(F_\vv )$-representations $(\pi_{\vv },V_{\vv })$, $\vv \notin S_\infty$, which are all unique up to isomorphism, such that
	$V_{(K_\infty)} \cong \bigotimes_{\vv \in S_\infty} V_{0,\vv } \ \otimes \ \rtprod_{\vv \notin S_\infty}V_\vv $ as $(\g_\infty,K_\infty,G(\A_f))$-modules. Hence, for each $n\in\N$, $V^{K_n}_{(K_\infty)}$ and  $\bigotimes_{\vv \in S_\infty} V_{0,\vv } \ \otimes \ \rtprod_{\vv \notin S_\infty}V^{K_{n,\vv }}_\vv $ are isomorphic (finitely generated, admissible) $(\g_\infty,K_\infty)$-modules, hence share the same Casselman-Wallach completion, 
	$$\overline{V^{K_n}_{(K_\infty)}}^{\sf \ CW} \cong \overline{\left(\bigotimes_{\vv \in S_\infty} V_{0,\vv } \ \otimes \ \rtprod_{\vv \notin S_\infty}V^{K_{n,\vv }}_\vv \right)}^{\sf \ CW},$$
	to a smooth admissible $G_\infty$-representation of moderate growth, cf.\ \cite{wallachII}, \S 11.5.6 and \S 11.6.8. Since $V^{K_n}$ is a Casselman-Wallach representation of $ G_\infty $, and $\rtprod_{\vv \notin S_\infty}V^{K_{n,\vv }}_\vv $ is finite-dimensional, we obtain an isomorphism of $G_\infty$-representations 
	$$V^{K_n} \cong  \overline{\left(\bigotimes_{\vv \in S_\infty} V_{0,\vv }\right)}^{\sf \ CW} \ \otimes \ \rtprod_{\vv \notin S_\infty}V^{K_{n,\vv }}_\vv ,$$
	cf.\ \cite{wallachII}, Thm.\ 11.6.7, which, according to \cite{vogan_venice}, Lem.\ 9.9.(3), simplifies to 
	$$V^{K_n} \cong \bigcprojtp_{\vv \in S_\infty} V_\vv \ \otimes \rtprod_{\vv \notin S_\infty}V^{K_{n,\vv }}_\vv ,$$
	for $(\pi_\vv ,V_\vv ):=(\overline{\pi_{0,\vv }}^{\sf \ CW},\overline{V_{0,\vv }}^{\sf \ CW}) $. Taking the inductive limit and applying \cite{warner}, Thm.\ A.2.2.5, we hence obtain an isomorphism of $G_\infty$-representations
	$$V\cong \bigcprojtp_{\vv \in S_\infty} V_\vv \ \cindtp \ \rtprod_{\vv \notin S_\infty}V_\vv $$
	where by construction, $\rtprod_{\vv \notin S_\infty}V_\vv $ carries the finest locally convex topology, cf.\ \S \ref{sect:LF}. Invoking the classical tensor product theorem once more, this isomorphism is seen to be $G(\A_f)$-equivariant on the dense subspace of $K_\infty$-finite vectors. Hence, the theorem follows from the continuity of this isomorphism and the continuity of the action of $G(\A_f)$.  
\end{proof}

%\begin{rem}
%	For real reductive groups (and their arithmetic subgroups) the concept of a smooth-automorphic subrepresentation may be extracted from \cite{wallach_smooth}, \S 6.1, where it has first been introduced (for groups of ``inner type''). Indeed, using the dictionary, Prop.\ \ref{prop:dict_umg}, the main result {\it ibidem} also shows that $V^{K_n}$ is a Casselman-Wallach representation of $ G_\infty $, if $V=\langle \vv arphi\rangle_{G(\A)}$ is generated by a single smooth-automorphic form. Over the adeles, the notion of a smooth-automorphic representation seems to appear first in \cite{cogdell_fields} for $G=\GL_n$. Our approach here for general connected reductive groups $G$, however, provides a more refined picture, in particular, in questions of topologies, which are partially left open in \cite{cogdell_fields}, even for $G=\GL_n$: Aside of our Rem.\ \ref{rem:001}, see, for example, Prop.\ \ref{prop:same_d} and its attached defining Eq.\ \eqref{eq:LF_smaut} for instances of such topological intricacies. Also the comparison in Thm.\ \ref{thm:005} and Thm.\  \ref{thm:TPthm} may serve as warning examples (or ``{\it tournants dangereux}''). 
%\end{rem}

\section{Main results: Parabolic and cuspidal support}

\subsection{}
In this section we will prove our main results on the decomposition of the space of smooth-automorphic forms along the parabolic and the more refined cuspidal support. \\\\ As far as the parabolic support of a smooth-automorphic form is concerned, our result provides a {\it topological} version of Langlands's algebraic direct sum decomposition of the space of smooth functions of uniform moderate growth, recorded and proved in two different ways in \cite{bls}, Thm.\ 2.4, and a little earlier in \cite{casselman_schwartz}, Thm.\ 1.16 and Cor.\ 4.7, respectively. We refer to \cite{langlands72} for the most original source.\\\\
On the other hand, our result on the cuspidal support formally mirrors the main result of \cite{franke_schwermer}, cf.\ Thm.\ 1.4.(2), on the level of smooth-automorphic forms. As a part of our analysis, we will prove a fine {\it structural, topological} decomposition of the space of cuspidal smooth-automorphic forms, which is a ``smooth'' analogue of the famous theorem of Gelfand--Graev--Piatetski-Shapiro \cite{GGPS},\S 7.2, on the decomposition of the space of cuspidal $L^2$-functions into a direct Hilbert sum. We refer to Thm.\ \ref{thm:cusp} below for this characterization of the irreducible cuspidal smooth-automorphic representations and to Thm.\ \ref{thm:cuspsup} for our main result on the cuspidal support of a general smooth-automorphic form.

\subsection{$A^\R_G$-invariant smooth-automorphic forms}
In order to obtain a meaningful notion of square-integrable smooth-automorphic forms, as usual, we shall pass over from $G(F)\backslash G(\A)$ to the smaller quotient $G(F) A^\R_G\backslash G(\A)\cong G(F)\backslash G(\A)^1$, which is of finite volume, i.e., we shall work with left-$G(F) A^\R_G$-invariant functions rather than with left-$G(F)$-invariant functions from now on. So, for every $ n\in\N $, let $\calA_d^\infty([G])^{K_n,\J^n}$ be the closed (and hence Fr\'echet) subspace of $A^\R_G$-invariant elements in $\calA_d^\infty(G)^{K_n,\J^n}$. We define the LF-space
\begin{equation}\label{eq:029}
\calA_\J^\infty([G]):=\lim_{n\to\infty}\calA_d^\infty([G])^{K_n,\J^n}.
\end{equation}
\begin{rem}
In the notation of Franke, cf.\ \cite{franke}, the underlying vector space of the LF-space $\calA_\J^\infty([G])$ was denoted $\mathfrak{Fin}_\J(C^\infty_{umg}(G \mathcal A_{\mathcal G}(\textit{\textbf{R}})^{o}\backslash \mathbb G))$. Franke, however, did not specify any topology on the vector space $\mathfrak{Fin}_\J(C^\infty_{umg}(G \mathcal A_{\mathcal G}(\textit{\textbf{R}})^{o}\backslash \mathbb G))$. 
\end{rem}
\noindent Observe that it is a priori by no means clear that the above LF-topology on $\calA_\J^\infty([G])$ agrees with the subspace-topology inherited from the inclusion $\calA_\J^\infty([G])\subseteq \calA_\J^\infty(G)$, i.e., that $\calA_\J^\infty([G])$ is in fact a smooth-automorphic subrepresentation of $\calA_\J^\infty(G)$. However, we shall see now that our ad hoc chosen LF-topology on $ \calA_\J^\infty([G]) $ behaves well with respect to the LF-topology on $ \calA_\J^\infty(G) $. To this end, recall that ${\rm Lie}(A^\R_G)=\a_G$. Since obviously $ \calA_\J^\infty([G])=\calA_{\J+\left<\a_G\right>_{\calZ(\g)}}^\infty([G]), $ when studying the spaces $ \calA_{\J}^\infty([G]) $, there is hence no loss of generality in assuming that $ \J\supseteq\a_G $. 
{\it We shall therefore suppose from now on that $ \J\supseteq\a_G $.} Doing so, we obtain 

\begin{prop}\label{prop:014}
	The space $ \calA_\J^\infty([G]) $, together with the right regular action, is a smooth-automorphic subrepresentation. Otherwise put, the LF-space topology on $ \calA_\J^\infty([G]) $, as defined by \eqref{eq:029}, agrees with the subspace topology inherited from $ \calA_\J^\infty(G) $. 
\end{prop}

\begin{proof}
	As indicated by the last sentence, we need to prove that $ \calA_\J^\infty([G]) $ is a topological subspace of $ \calA_\J^\infty(G) $. Since for every $ n\in\N $, $ \calA_d^\infty([G])^{K_n,\J^n} $ is a closed topological subspace of $ \calA_d^\infty(G)^{K_n,\J^n} $, the inclusion map provides a continuous embedding $ \calA_\J^\infty([G])=\lim_n\calA_d^\infty([G])^{K_n,\J^n}\hookrightarrow\lim_n\calA_d^\infty(G)^{K_n,\J^n}=\calA_\J^\infty(G) $. To finish the proof, we show that there exists a continuous retraction $\calA_\J^\infty(G)\to\calA_\J^\infty([G]) $, the continuity here being the subtle point:\\\\ In order to construct such a retraction, let us fix a Haar measure $ da $ on $ A_G^\R $. Since $ G(\A)=A_G^\R\times G(\A)^1 $, it follows from \cite{beuzart}, Prop.\ A.1.1(vi), using \eqref{eq:ungl}, that there exists $ r\in\N $ such that $ \int_{A_G^\R}\norm a^{-r}\,da<\infty $. 
Let $ d_0:=d+r\in\N $ and $ D_0:=\int_{A_G^\R}\norm a^{-d_0}\,da\in\R_{>0} $. We claim that the linear operator $ P_{A_G^\R,d_0}:\calA_\J^\infty(G)\to\calA_\J^\infty([G]) $, given by
	\[ \left(P_{A_G^\R,d_0}(\varphi)\right)(a'g):=D_0^{-1}\int_{A_G^\R}\varphi(ag)\,\norm a^{-d_0}\,da,\qquad a'\in A_G^\R,\ g\in G(\A)^1, \]
	is a well-defined, continuous retraction as desired. Indeed, using \eqref{eq:ungl} and \cite[Prop.\ A.1.1(i)]{beuzart}, one directly checks that for every $ n\in\N $ and $ \varphi\in\calA_d^\infty(G)^{K_n,\J^n} $, $ P_{A_G^\R,d_0}(\varphi) $ is a well-defined left-$ G(F)A_G^\R $-invariant, right-$ K_n $-invariant, smooth function $ G(\A)\to\C $ satisfying
	\begin{equation}\label{eq:017}
	XP_{A_G^\R,d_0}(\varphi)=P_{A_G^\R,d_0}(X_{\g_\infty^1}\varphi),\qquad X\in\U(\g),
	\end{equation}
	where, denoting $ G_\infty^1:=G_\infty\cap G(\A)^1 $, $ X_{\g_\infty^1} $ is the projection of $ X $ on $ \U(\g_\infty^1) $ with respect to the direct sum decomposition $ \U(\g)=\left<\a_G\right>_{\U(\g)}\oplus\U(\g_\infty^1) $. Since the latter decomposition restricts to the decomposition $ \calZ(\g)=\left<\a_G\right>_{\calZ(\g)}\oplus\calZ(\g_\infty^1) $, and $ \J\supseteq\a_G $, we have that $ \J=\left<\a_G\right>_{\calZ(\g)}\oplus\left(\J\cap\calZ(\g_\infty^1)\right) $, hence
	\[ \J^nP_{A_G^\R,d_0}(\varphi)\overset{\eqref{eq:017}}=P_{A_G^\R,d_0}\left(\left(\left(\left<\a_G\right>_{\calZ(\g)}\oplus\left(\J\cap\calZ(\g_\infty^1)\right)\right)^n\right)_{\g_\infty^1}\varphi\right)=P_{A_G^\R,d_0}\left(\left(\J\cap\calZ(\g_\infty^1)\right)^n \varphi\right)=0. \]
	Finally, by \cite[I.2.2(viii)]{moewal} and \eqref{eq:ungl} there exist $ M\in\R_{>0} $ and $ t\in\N $ such that
	\begin{equation}\label{eq:018}
	\norm g\leq M\norm{ag}^t,\qquad a\in A_G^\R,\ g\in G(\A)^1,
	\end{equation}
	hence we obtain the estimate
	\[ \begin{aligned}
	p_{td,X}(P_{A_G^\R,d_0}(\varphi))
	&\overset{\eqref{eq:017}}=\sup_{\substack{g\in G(\A)^1\\a'\in A_G^\R}}\left|\left(P_{A_G^\R,d_0}(X_{\g_\infty^1}\varphi)\right)(a'g)\right|\,\norm{a'g}^{-td}\\
	&\overset{\eqref{eq:018}}\leq\sup_{g\in G(\A)^1}D_0^{-1}\int_{A_G^\R}\left|\left(X_{\g_\infty^1}\varphi\right)(ag)\right|\,\norm{a}^{-d_0}\,da\ M^d\norm g^{-d}\\
	&\overset{\phantom{\eqref{eq:ungl}}}\leq\frac{M^d}{D_0}\sup_{g\in G(\A)^1}\int_{A_G^\R}p_{d,X_{\g_\infty^1}}(\varphi)\,\norm{ag}^d\,\norm a^{-d_0}\,\norm g^{-d}\,da\\	
	&\overset{\eqref{eq:ungl}}\leq\frac{(MC_0)^d}{d_0}\,\int_{A_G^\R}\norm a^{-r}\,da\ \cdot p_{d,X_{\g_\infty^1}}(\varphi),\qquad X\in\U(\g),\ \varphi\in\calA_d^\infty(G)^{K_n,\J^n}.
	\end{aligned} \]
	This proves that for every $ n\in\N $, $ P_{A_G^\R,d_0} $ restricts to a well-defined, continuous operator $ \calA_d^\infty(G)^{K_n,\J^n}\to\calA_{td}^\infty([G])^{K_n,\J^n} $, which obviously fixes every $ \varphi\in\calA_d^\infty([G])^{K_n,\J^n} $. Going over to the direct limit, it follows that $ P_{A_G^\R,d_0} $ is a well-defined, continuous retraction
	\[ \calA^\infty_\J(G)=\lim_{n\to\infty}\calA^\infty_d(G)^{K_n,\J^n}\to\lim_{n\to\infty}\calA^\infty_{td}([G])^{K_n,\J^n}=\calA^\infty_\J([G]), \]
	where the last equality holds by our Prop.\ \ref{prop:same_d} and Rem.\ \ref{rem:018} (which applies equally well to $\calA^\infty_\J([G])$).
\end{proof}

\subsection{LF-compatible smooth-automorphic representations and their direct sums}
It will be convenient to introduce the following notion: We will call a smooth-automorphic $ G(\A) $-representation $ (\pi,V) $ \textit{LF-compatible}, if $ V=\lim_{n\to\infty}V^{K_n,\J^n} $ topologically. Here we used the suggestive notation $V^{K_n,\J^n}$ to indicate the closed $ G_\infty $-invariant subspace of $ K_n $-invariant elements of $ V $ that are annihilated by $ \J^n $. As we have just observed, $ \calA_\J^\infty([G]) $ is an LF-compatible smooth-automorphic representation, as well as every smooth-automorphic $ G(\A) $-representation that is annihilated by a power of $ \J $, hence in particular every irreducible smooth-automorphic representation. Obviously, if $ (\pi,V) $ is an LF-compatible smooth-automorphic subrepresentation, then $ V=\lim_{n\to\infty}V^{K_n,\J^n} $ is an LF-space.\\\\
We record the following lemma, which is based on two results of Harish-Chandra and the fact that $ \calA^\infty_d(G)^{K_n,\J^n} $ is a Casselman-Wallach representation: 

\begin{lem}\label{lem:021}
	Let $ n\in\N $, and let $ \left\{V_{0,i}\right\}_{i\in I} $ be a family of $ (\g_\infty,K_\infty) $-submodules of $ \calA_\J(G)^{K_n,\J^n} $ whose sum is direct. Then, the sum of $ G_\infty $-invariant subspaces $ \overline{V_{0,i}}:=\Cl_{\calA_d^\infty(G)^{K_n,\J^n}}(V_{0,i}) $ of $ \calA_d^\infty(G)^{K_n,\J^n} $ is also direct.
\end{lem}
\begin{proof}
	For every $ i\in I $, by Prop.\ \ref{prop:sm_inf} and \cite{varada}, Thm.\ II.7.14, $ \overline{V_{0,i}} $ is a $ G_\infty $-invariant subspace of $ \calA_d^\infty(G)^{K_n,\J^n} $, and $ \left(\overline{V_{0,i}}\right)_{(K_\infty)}=V_{0,i} $. Now, suppose that $ \sum_{i\in I}\phi_i=0 $ for some $ \phi_i\in \overline{V_{0,i}} $, where $ \phi_i=0 $ for all but finitely many $ i $. Then, for every irreducible representation $ \rho_\infty$ of ${K_\infty} $, $ \sum_{i\in I}E_{\rho_\infty}(\phi_i)=0 $. Since $ E_{\rho_\infty}(\phi_i)\in \left(\overline{V_{0,i}}\right)_{(K_\infty)}=V_{0,i} $, and the sum of $ V_{0,i} $'s is direct, it follows that $ E_{\rho_\infty}(\phi_i)=0 $ for all $ {\rho_\infty} $ and $ i\in I$. Thus, since $ \calA_d^\infty(G)^{K_n,\J^n} $ is a smooth representation of $ K_\infty $, by \cite[Lem.\ 5]{hch66} we have that $ \phi_i=\sum_{\rho_\infty}E_{\rho_\infty}(\phi_i)=0 $ for every $ i\in I $.
\end{proof}

The next result on LF-compatible smooth-automorphic subrepresentations will be crucial.

\begin{prop}\label{prop:108}
	Let $ V $ be an LF-compatible smooth-automorphic subrepresentation, and denote $ V_0:=V_{(K_\infty)} $. Suppose that 
	\begin{equation}\label{eq:101}
	V_0=\bigoplus_{i\in I}V_{0,i}
	\end{equation}
	for some $ (\g_\infty,K_\infty,G(\A_f)) $-submodules $ V_{0,i}\subseteq V_0 $, and denote $ V_i:=\Cl_{\calA^\infty_\J(G)}(V_{0,i}) $. Then, we have the following decomposition into a locally convex direct sum of LF-compatible smooth-automorphic subrepresentations:
	\begin{equation}\label{eq:105}
	V=\bigtoplus_{i\in I}V_i.
	\end{equation}
	In particular, for every $ i\in I $, we have
	\begin{equation}\label{eq:106}
	(V_i)_{(K_\infty)}=V_{0,i},
	\end{equation}
	hence, if the $ (\g_\infty,K_\infty,G(\A_f)) $-module $ V_{0,i} $ is irreducible, then so is the $ G(\A) $-representation $ V_i $. Moreover,
	\begin{equation}\label{eq:104}
	V_i=\lim_{n\to\infty}V_i^{K_n,\J^n}=\lim_{n\to\infty}\overline{V_{0,i}^{K_n,\J^n}},
	\end{equation}
	where $ \overline{V_{0,i}^{K_n,\J^n}}:=\Cl_{\calA_d^\infty(G)^{K_n,\J^n}}\left(V_{0,i}^{K_n,\J^n}\right) $. For every fixed $ n\in\N $,
	\begin{equation}\label{eq:107}
	V_i^{K_n,\J^n}=V_{0,i}^{K_n,\J^n}=0\qquad\text{for all but finitely many }i\in I.
	\end{equation}
\end{prop}  

\begin{proof}
	Let us first show that for every fixed $ n\in\N $,
	\begin{equation}\label{eq:100}
		V_{0,i}^{K_n,\J^n}=0\qquad\text{for all but finitely many }i\in I.	
	\end{equation}
	Since by \eqref{eq:101}
	\begin{equation}\label{eq:102}
		V_0^{K_n,\J^n}=\bigoplus_{i\in I}V_{0,i}^{K_n,\J^n},
	\end{equation}
	it suffices to prove that the $ (\g_\infty,K_\infty) $-module $ V_0^{K_n,\J^n} $ is finitely generated, which holds by \cite{vogan}, Cor.\ 5.4.16, since $ V_0^{K_n,\J^n}\subseteq\calA_\J(G)^{K_n,\J^n} $ is obviously $ \calZ(\g) $-finite and is admissible by \cite{bojac}, \S 4.3(i).
	
	Next, by Lem.\ \ref{lem:008}, $ V_i $ is a smooth-automorphic subrepresentation for every $ i\in I $. To prove the proposition, it remains to prove that 
	\begin{equation}\label{eq:103}
	V=\bigtoplus_{i\in I}\lim_{n\to\infty}\overline{V_{0,i}^{K_n,\J^n}}.
	\end{equation}
	Indeed, \eqref{eq:103} implies that $ V_i=\lim_{n\to\infty}\overline{V_{0,i}^{K_n,\J^n}} $ for every $ i\in I $ and that \eqref{eq:105} holds; \eqref{eq:105} and \eqref{eq:101} imply \eqref{eq:106}; Thm.\ \ref{thm:005} -- or, alternatively, a combination of \eqref{eq:106} and \cite{hch66}, Lem.\ 4 -- implies that $ V_i $ is irreducible whenever $ V_{0,i} $ is; moreover, \eqref{eq:106} implies that for every $ n\in\N $, $ \left(V_i^{K_n,\J^n}\right)_{(K_\infty)}=V_{0,i}^{K_n,\J^n} $, hence by \cite{hch66}, Lem.\ 4, $ V_i^{K_n,\J^n}=\overline{V_{0,i}^{K_n,\J^n}} $, which finishes the proof of \eqref{eq:104} and, by \eqref{eq:100}, \eqref{eq:107}.

	Therefore, we are left to prove \eqref{eq:103}. The locally convex direct sum on the right-hand side of \eqref{eq:103} is well-defined since by Lem.\ \ref{lem:021}, for every $ n\in\N $ the sum of subspaces $ \overline{V_{0,i}^{K_n,\J^n}}\subseteq\calA_d^\infty(G)^{K_n,\J^n} $ is direct. Next, we have
	\begin{equation}\label{eq:109}
	\bigtoplus_{i\in I}\lim_{n\to\infty}\overline{V_{0,i}^{K_n,\J^n}}=\lim_{n\to\infty}\bigtoplus_{i\in I}\overline{V_{0,i}^{K_n,\J^n}},
	\end{equation}
	which is obvious as an equality of vector spaces, and the topology of both sides is easily seen to be the finest locally convex topology with respect to which the inclusion maps of the subspaces $ \overline{V_{0,i}^{K_n,\J^n}} $ are continuous (a detail, which we leave to the reader). To prove \eqref{eq:103}, by \eqref{eq:109} and the LF-compatibility of $ V $ it suffices to prove that for every $ n\in\N $,
	\begin{equation}\label{eq:110}
	V^{K_n,\J^n}=\bigtoplus_{i\in I}\overline{V_{0,i}^{K_n,\J^n}}.
	\end{equation}
	By Prop.\ \ref{prop:sm_inf}, \cite{wallachII}, Lem.\ 11.5.2, and \cite{varada}, Thm.\ II.7.14, for every $ i\in I $, $ \overline{V_{0,i}^{K_n,\J^n}} $ is a Casselman-Wallach representation of $ G_\infty $ and $ \left(\overline{V_{0,i}^{K_n,\J^n}}\right)_{(K_\infty)}=V_{0,i}^{K_n,\J^n} $.
	By \cite{hch66}, Lem.\ 4, \eqref{eq:100} and \eqref{eq:102}, it follows that $ \bigtoplus_{i\in I}\overline{V_{0,i}^{K_n,\J^n}} $ is a Casselman-Wallach representation of $ G_\infty $ having $ V_0^{K_n,\J^n} $ as its $ (\g_\infty,K_\infty) $-module of $ K_\infty $-finite vectors. Since by Prop.\ \ref{prop:sm_inf} and \cite{wallachII}, Lem.\ 11.5.2, the same holds for $ V^{K_n,\J^n} $, by \cite{wallachII}, Thm.\ 11.6.7(2) the identity map on $ V_0^{K_n,\J^n} $ extends to a $ G_\infty $-equivalence
	\[ \iota:\bigtoplus_{i\in I}\overline{V_{0,i}^{K_n,\J^n}}\to V^{K_n,\J^n}. \]
	By \cite{hch66}, Lem.\ 4 (and the uniqueness of continuous extensions from a dense subspace), $ \iota $ coincides with the (obviously continuous) inclusion map, hence $ \iota $ is the identity map. This proves \eqref{eq:110} and hence the proposition.

\end{proof}

\subsection{The parabolic support of a smooth-automorphic form}
Let $ \calA_{cusp}^\infty([G]) $ denote the space of functions $ \varphi\in\calA^\infty([G]):=\bigcup_{\J}\calA_\J^\infty([G]) $ that are \textit{cuspidal}, i.e., satisfy $ \varphi_P=0 $ for every proper parabolic $ F $-subgroup $ P $ of $ G $.
We say a function $f\in C^\infty_{umg}(G(F)\backslash G(\A))$ is \textit{negligible} along a parabolic $ F $-subgroup $Q=LN$ of $ G $, if
\begin{equation}\label{eq:113}
\lambda_{Q,g,\varphi}(f):=\int_{L(F)\backslash L(\A)^1} f_Q(lg) \,\overline{\varphi(l)} \,dl =0,\qquad \forall g\in G(\A),\ \forall \varphi\in \calA_{cusp}^\infty([L]).
\end{equation}

\begin{rem}\label{rem:negl}
Usually negligibility is defined as an orthogonality-relation with respect to all cuspidal automorphic forms, i.e., one supposes that \eqref{eq:113} only holds for all $\phi\in \calA_{cusp}([L])$. See \cite{bls}, \S 2.2, \cite{borel_co}, \S 6.7, \cite{franke_schwermer}, \S 1.1 and (most explicitly) \cite{osborne-warner}, p.\ 82 (which refers to Langlands's original work \cite{langlandsLNM}). In course of proving Thm.\ \ref{thm:112} below, we will show that our definition is in fact equivalent to the common one, but has the advantage to be intrinsic to the notion of smooth-automorphic forms. 
\end{rem}
For every $ P\in\calP $, let $ \left\{P\right\} $ denote the associate class of $ P $, i.e., the set of parabolic $ F $-subgroups $ Q=L_QN_Q $ of $ G $ such that $ L_Q $ is conjugate to $ L_P $ by an element of $ G(F) $. We define $ \calA^\infty_{\J,\P}([G])$ to be the space of smooth-automorphic forms $\varphi\in \calA^\infty_\J([G])$, which are negligible along all $Q\notin\P$. The following theorem is our first main result:

\begin{thm}\label{thm:112}
	We have the following, $G(\A)$-equivariant decomposition into a locally convex direct sum of LF-compatible smooth-automorphic subrepresentations:
	\[ \calA^\infty_\J([G])=\bigtoplus_{\left\{P\right\}}\calA^\infty_{\J,\P}([G]). \]
\end{thm}

\begin{proof}
We proceed in several steps.\\
{\it Step 1}: Let $K_{L,n}:=K_n\cap L(\A_f)$. For given $ f\in C^\infty(L(F)A^\infty_Q\backslash L(\A)) $, $ d\in\Z $ and $ X\in\U(\l) $, let us abbreviate
 \[ q_{d,X}(f):=\sup_{l\in L(\A)^1}\absl{(Xf)(l)}\,\norm l_{[L]}^{-d} \]
where $ \norm l_{[L]}:=\inf_{\gamma\in L(F)}\norm{\gamma l} $ and define $ \mathcal S([L])^{K_{L,n}} $ to be the space, which consists of all functions $ f\in C^\infty(L(F)A_Q^\R\backslash L(\A))^{K_{L,n}} $ such that  $q_{d,X}(f)<\infty$ for all $ d\in\Z $ and $ X\in\U(\l) $. We equip $ \mathcal S([L])^{K_{L,n}} $ with the Fréchet topology generated by the seminorms $q_{d,X}$ and define the global Schwartz space
\[ \mathcal S([L]):=\lim_{n\to\infty}\mathcal S([L])^{K_{L,n}}. \]
As for any compact set $ C\subseteq N(\A) $ such that $ N(\A)=N(F)C $ and $d=d_0\in \N$ as in \S \ref{sect:LF-autom}
\[ \absl{f_Q(lg)}\leq\int_C\absl{f(nlg)}\,dn\leq p_{d,1}(f)\,\int_C\norm{nlg}^d\,dn \overset{\ref{eq:ungl}}\leq C_0^{2d}\ p_{d,1}(f)\,\norm l^d\,\norm g^d\,\int_C\norm n^d\,dn,\quad l\in L(\A)^1, \]
	we have
	\[ \begin{aligned}
	\int_{L(F)\backslash L(\A)^1} \absl{f_Q(lg) \,\varphi(l)} \,dl	
	&\leq C_0^{2d}\,p_{d,1}(f)\,\norm g^d\,\int_C\norm n^d\,dn\,\int_{L(F)\backslash L(\A)^1}\absl{\varphi(l)}\norm{l}_{[L]}^d\,dl\\
	&\leq \left(C_0^{2d}\,\norm g^d\,\int_C\norm n^d\,dn\,\vol(L(F)\backslash L(\A)^1)\right)\ \,p_{d,1}(f) \,q_{-d,1}(\varphi)
	\end{aligned} \]
for all $ f\in \calA^\infty_\J([G])$ and for all $\varphi\in \mathcal S([L])$. Hence, we have shown that for each fixed $g\in G(\A)$, the function $\lambda_{Q,g,\varphi}(f)$ is separately continuous in the the arguments $ f\in \calA^\infty_\J([G])$ and $\varphi\in \mathcal S([L])$.\\\\
{\it Step 2}: As announced in Rem.\ \ref{rem:negl}, we will now show that our definition of negligibility coincides with the usual (weaker) one, i.e., we will prove that a $ f\in C^\infty_{umg}(G(F)\backslash G(\A)) $ is negligible along a parabolic $ F $-subgroup $Q=LN$ of $ G $ if and only if $\lambda_{Q,g,\phi}(f)=0$ for all $g\in G(\A)$ and all $\phi\in\calA_{cusp}([L])$. Necessity being obvious, we show sufficiency: Suppose $\lambda_{Q,g,\phi}(f)=0$ for all $g\in G(\A)$ and all $\phi\in\calA_{cusp}([L])$, i.e., $ \calA_{cusp}([L])\subseteq\ker\lambda_{Q,g,\bullet}(f) $ for every $ g\in G(\A) $. By continuity of $ \lambda_{Q,g,\bullet}(f): \mathcal S([L])\ra \C $, as shown in Step 1 above, it follows that $ \Cl_{\mathcal S([L])}(\calA_{cusp}([L]))\subseteq\ker\lambda_{Q,g,\bullet}(f)$ for every $ g\in G(\A) $. Thus, it suffices to prove that for every $ g\in G(\A) $,
	\begin{equation}\label{eq:302}
	\mathcal A^\infty_{cusp}([L])\subseteq\Cl_{\mathcal S([L])}(\calA_{cusp}([L])). 
	\end{equation}
	To this end, note that for every ideal $ \J $ in $ \calZ(\l) $ and $ n\in\N $, we have an equality of Fréchet spaces
	\begin{equation}\label{eq:303}
	\calA^\infty_{cusp,\J}([L])^{K_{L,n},\J^n}=\mathcal S_{cusp}([L])^{K_{L,n},\J^n},
	\end{equation}
$ \mathcal S_{cusp}([L])^{K_{L,n},\J^n}$ (resp.\ $\calA^\infty_{cusp,\J}([L])^{K_{L,n},\J^n}$) denoting the subspace of cuspidal, $\J^n$-annihilated functions in $ \mathcal S([L])^{K_{L,n}}$ (resp.\ $\calA^\infty_{\J}([L])^{K_{L,n}}$). Indeed, it is a simple consequence of \cite{moewal}, Cor.\ I.2.11 and I.2.2.(vi), {\it ibidem}, that these two spaces coincide as sets. Moreover, for every parabolic $ F $-subgroup $ Q'=L'N' $ of $ L $, $ l\in L(\A) $ and $ Y\in\J^n $, the inequalities 

\begin{itemize}
	\item[] $ \displaystyle \absl{f_{Q'}(l)}\leq\int_{N'(F)\backslash N'(\A)}\absl{f(n'l)}\,dn'\leq \norm f_\infty=q_{0,1}(f) $ for all $ f\in\mathcal S([L])^{K_{L,n}} $
	\item[] $ \displaystyle q_{d,X}(Yf)=\sup_{l\in L(\A)^1}\absl{(XYf)(l)}\norm l_{[L]}^{-d}=q_{d,XY}(f) $ for all $ d\in\Z $, $ X\in\U(\l) $ and $ f\in\mathcal S([L])^{K_{L,n}} $	
\end{itemize}
imply that the assignments $ f\mapsto f_{Q'}(l) $ and $ f\mapsto Yf $ define continuous linear operators $ \mathcal S([L])^{K_{L,n}}\to\C $ and $ \mathcal S([L])^{K_{L,n}}\to\mathcal S([L])^{K_{L,n}} $, respectively. Thus, the intersection $ \mathcal S_{cusp}([L])^{K_{L,n},\J^n} $ of their kernels is a closed subspace of $ \mathcal S([L])^{K_{L,n}} $ and hence Fr\'echet. As the identity map $\mathcal S_{cusp}([L])^{K_{L,n},\J^n}\ra \calA^\infty_{cusp,\J}([L])^{K_{L,n},\J^n}$ is obviously continuous, the claimed equality \eqref{eq:303} of Fr\'echet spaces finally follows from the open mapping theorem. Thus,  writing $\calA_{cusp,\J}([L])^{K_{L,n},\J^n}:=\calA^\infty_{cusp,\J}([L])_{(K_\infty)}^{K_{L,n},\J^n}$, we have
	\[ \begin{aligned}
	\mathcal A^\infty_{cusp}([L])
	&=\bigcup_{n,\J}\calA^\infty_{cusp,\J}([L])^{K_{L,n},\J^n}\\
	&=\bigcup_{n,\J}\Cl_{\calA^\infty_{cusp,\J}([L])^{K_{L,n},\J^n}}(\calA_{cusp,\J}([L])^{K_{L,n},\J^n})\\	
	&\overset{\eqref{eq:303}}=\bigcup_{n,\J}\Cl_{\mathcal S_{cusp}([L])^{K_{L,n},\J^n}}(\calA_{cusp,\J}([L])^{K_{L,n},\J^n})\\		
	&\subseteq\bigcup_{n,\J}\Cl_{\mathcal S([L])}(\calA_{cusp,\J}([L])^{K_{L,n},\J^n})\\		
	&\subseteq\Cl_{\mathcal S([L])}\left(\bigcup_{n,\J}\calA_{cusp,\J}([L])^{K_{L,n},\J^n}\right)\\			
	&=\Cl_{\mathcal S([L])}(\calA_{cusp}([L])).
	\end{aligned} \]
This proves \eqref{eq:302} and hence that a function $ f\in C^\infty_{umg}(G(F)\backslash G(\A)) $ is negligible along a parabolic $ F $-subgroup $Q=LN$ of $ G $ if and only if $\lambda_{Q,g,\phi}(f)=0$ for all $g\in G(\A)$ and all $\phi\in\calA_{cusp}([L])$.\\\\
{\it Step 3}: Step 2 now allows us to use Langlands's algebraic direct sum decomposition
\begin{equation}\label{eq:111}
C^\infty_{umg}(G(F)\backslash G(\A))=\bigoplus_{\left\{P\right\}}C^\infty_{umg,\P}(G(F)\backslash G(\A)),
\end{equation}
where $C^\infty_{umg,\P}(G(F)\backslash G(\A))$ denotes the space of functions $f\in C^\infty_{umg}(G(F)\backslash G(\A))$ that are negligible along all $Q\notin\P$. See \cite{bls}, Thm.\ 2.4, or \cite{casselman_schwartz}, Thm.\ 1.16 and Cor.\ 4.7 for a proof. Since the spaces $C^\infty_{umg,\P}(G(F)\backslash G(\A))$ are invariant under the action of $ G(\A) $ and $ \U(\g) $ by right translation, \eqref{eq:111} implies that
\begin{equation}\label{eq:115}
\calA^\infty_\J([G])=\bigoplus_{\left\{P\right\}}\calA^\infty_{\J,\P}([G])
\end{equation}
as a $G(\A)$-equivariant decomposition of vector spaces and
\begin{equation}\label{eq:112}
 \calA_\J([G])=\bigoplus_{\left\{P\right\}}\calA_{\J,\P}([G]), 
\end{equation}
as a decomposition of $ (\g_\infty,K_\infty,G(\A_f)) $-modules, where $ \calA_{\J,\P}([G]):=\calA_{\J,\P}^\infty([G])_{(K_\infty)} $. See also \cite{franke_schwermer}, 1.1. It hence follows from our Prop.\ \ref{prop:108} that we have the following decomposition into a locally convex direct sum of LF-compatible smooth-automorphic subrepresentations:
	\begin{equation}\label{eq:114}
	\calA^\infty_\J([G])=\bigtoplus_{\left\{P\right\}}\Cl_{\calA^\infty_\J([G])}(\calA_{\J,\P}([G])).
	\end{equation}
Therefore, it remains to prove that for every $ \left\{P\right\} $, $ \Cl_{\calA^\infty_\J([G])}(\calA_{\J,\P}([G]))=\calA^\infty_{\J,\left\{P\right\}}([G]) $. 
By \eqref{eq:114} and \eqref{eq:115} it suffices to prove that $ \Cl_{\calA^\infty_\J([G])}(\calA_{\J,\P}([G]))\subseteq\calA^\infty_{\J,\left\{P\right\}}([G]) $. We recall from Step 1 above that the linear functional $ \lambda_{Q,g,\varphi}:\calA^\infty_\J([G])\to\C $ is continuous for every $ F $-parabolic subgroup $ Q=LN $ of $ G $, $ g\in G(\A) $ and $\varphi\in \mathcal S([L])$, hence, by \eqref{eq:303}, in particular for every $ \varphi\in\calA^\infty_{cusp}([L]) \subseteq \mathcal S([L])$. Therefore, $ \calA_{\J,\P}^\infty([G]) $ is closed in $ \calA^\infty_\J([G]) $. This shows the claim.
\end{proof}

\subsection{Cuspidal smooth-automorphic forms}
We will now consider the subspace $ \calA^\infty_{cusp,\J}([G]) $ of cuspidal functions in $ \calA_{\J}^\infty([G]) $ more closely. It is well-known (cf.\ combine Step 2 in the proof of Thm.\ \ref{thm:112} with \cite{bls}, Prop.\ 2.3%, or with \cite{borel_co}, \S 6.7
) that 
$$\calA^\infty_{cusp,\J}([G]) = \calA^\infty_{\J,\left\{G\right\}}([G]).$$
Thus, Thm.\ \ref{thm:112} implies the following

\begin{cor}\label{cor:111}
	The space of cuspidal smooth-automorphic forms $ \calA_{cusp,\J}^\infty([G]) $ is an LF-compatible smooth-automorphic subrepresentation.
\end{cor}
It is the goal of this subsection to give a refined description of the smooth-automorphic $G(\A)$-representation $ \calA_{cusp,\J}^\infty([G]) $ as a countable locally convex direct sum of irreducible subrepresentations. \\\\ 
First, to settle terminology, we will call a smooth-automorphic representation {\it cuspidal}, if it is isomorphic to a subquotient of $\calA_{cusp,\J}^\infty([G])$ (for some ideal $\J$ of $\calZ(\g)$ of finite codimension). Now, let $L_{cusp}^2([G])$ be the Hilbert space of classes of left-$G(F)A^\R_G$-invariant cuspidal functions $f: G(\A)\ra\C$, which are square-integrable modulo $G(F)A^\R_G$, endowed with the $L^2$-norm. As it is well-known (following from \ \cite{GGPS}, \S 7.2), $L_{cusp}^2([G])$ is a unitary representation of $G(\A)$ by right translation, which decomposes as a countable direct Hilbert sum of irreducible subrepresentations $\mathcal H_i$,
\begin{equation}\label{eq:l2}
L_{cusp}^2([G]) =\widehat\bigoplus_{i\in I} \ \mathcal H_i,
\end{equation}
each of which appearing with finite multiplicity $m(\mathcal H_i):=\dim \Hom_{G(\A)}(\mathcal H_i, L_{cusp}^2([G]))$ (which may grow unboundedly, though, as $\mathcal H_i$ varies through $i\in I$, cf.\ \cite{ggj}, (1.1)). It is finally a consequence of \cite{borel_localement}, Prop.\ 5.26, that hence there exists a (unique) subset $I(\J)\subseteq I$ such that we have the following decomposition of $ \calA_{cusp,\J}([G]):=\calA_{cusp,\J}^\infty([G])_{(K_\infty)} $ into a countable algebraic direct sum of irreducible $ (\g_\infty,K_\infty,G(\A_f)) $-modules:
\begin{equation}\label{eq:117}
\calA_{cusp,\J}([G])\cong \bigoplus_{i\in I(\J)} \ \mathcal H^{\infty_\A}_{i,(K_\infty)}.
\end{equation}
Here we identify each element of the right-hand side, which is by definition an equivalence class of almost everywhere equal measurable functions on $ G(F)A_G^\R\backslash G(\A) $, with its unique continuous representative. \\\\
On the level of cuspidal smooth-automorphic forms, it easily follows from \cite{borel_co},  Eq.\ 6.8.4 in combination with our Prop.\ \ref{prop:dict_umg}, that we have an identification of vector spaces 
$$L^2_{cusp,\J}([G])^{\infty_\A}_{(\mathcal Z(\g))} \cong \calA_{cusp,\J}^\infty([G]),$$
of the space of smooth, $\mathcal Z(\g)$-finite vectors\footnote{Unlike stated in the literature, %\cite{cogdell_fields}, \S 3.4, 
the additional assumption of $\mathcal Z(\g)$-finiteness is essential as $L^2_{cusp,\J}([G])^{\infty_\A} \neq \calA_{cusp,\J}^\infty([G])$.} %vgl Sonjas Notes  
in $L_{cusp,\J}^2([G]):=\widehat\bigoplus_{i\in I(\J)} \ \mathcal H_i$ and the space of cuspidal smooth-automorphic forms, given by assigning each class in $L^2_{cusp,\J}([G])^{\infty_\A}_{(\mathcal Z(\g))}$ its unique continuous representative. \\\\
It is by no means clear, however, that this identification is compatible with the direct sum decomposition of $L_{cusp,\J}^2([G])$, i.e., it is not clear that the LF-spaces of globally smooth vectors in the irreducible subrepresentations $\mathcal H_i$, $i\in I(\J)$, of the Hilbert space representation $L^2_{cusp}([G])$ identify with irreducible subrepresentations of the LF-space $\calA_{cusp,\J}^\infty([G])$. It is the goal of this section to establish the following

\begin{thm}\label{thm:cusp}
The isomorphism \eqref{eq:117} extends to a $G(\A)$-equivariant decomposition into a countable locally convex direct sum of LF-compatible smooth-automorphic subrepresentations:
	\[ \calA^\infty_{cusp,\J}([G])\cong\bigtoplus_{i\in I(\J)}\calH_i^{\infty_\A}. \]
Consequently, for each $i\in I(\J)$, $\mathcal H^{\infty_\A}_i=\lim_n (\mathcal H^{\infty_\R}_i)^{K_n}$ with its natural LF-space topology embeds as a $G(\A)$-subrepresentation into $\calA_{cusp,\J}^\infty([G])$. This characterises the irreducible cuspidal smooth-automorphic subrepresentations of $\calA_\J^\infty(G)$ as the subrepresentations isomorphic to the LF-spaces $\mathcal H^{\infty_\A}_i$ of smooth vectors in the unitary Hilbert space representations $\mathcal H_i$, $i\in I(\J)$.
\end{thm}

\begin{proof}
For each $i\in I$, let us write $\theta_i$ for the linear map assigning each class $[f]\in\calH_{i}^{\infty_\A}$ its unique continuous representative and let us abbreviate $H_{i,(K_\infty)}^{\infty_\A}:=\theta_i(\calH_{i,(K_\infty)}^{\infty_\A})$. Then, combining Prop.\ \ref{prop:014}, Cor.\ \ref{cor:111} and \eqref{eq:117} with Prop.\ \ref{prop:108} implies that we have the following decomposition of the $G(\A)$-representation $\calA^\infty_{cusp,\J}([G])$ into a locally convex direct sum of irreducible LF-compatible smooth-automorphic subrepresentations:
	\[ \calA^\infty_{cusp,\J}([G])=\bigtoplus_{i\in I(\J)}\Cl_{\calA_\J^\infty([G])}(H_{i,(K_\infty)}^{\infty_\A}). \]
We now prove that for every $ i\in I(\J) $ the map $\theta_i$ extends to an isomorphism $\calH_i^{\infty_\A} \cong \Cl_{\calA_\J^\infty([G])}(H_{i,(K_\infty)}^{\infty_\A})$ of $G(\A)$-representations. As we have just observed, for every $ i\in I(\J)$, $ \Cl_{\calA_\J^\infty([G])}(H_{i,(K_\infty)}^{\infty_\A}) $ is an irreducible smooth-automorphic subrepresentation. Hence, by Prop.\ \ref{prop:smmf} each $ \Cl_{\calA_\J^\infty([G])}(H_{i,(K_\infty)}^{\infty_\A}) $ is a Casselman-Wallach representation of $ G(\A) $, which has $H_{i,(K_\infty)}^{\infty_\A} $ as its $ (\g_\infty,K_\infty,G(\A_f)) $-module of $ K_\infty $-finite vectors, cf.\ Thm.\ \ref{thm:005} and Prop.\ \ref{prop:014}. Similarly, by \cite{wallachII}, Lemmas 11.5.1 and 11.5.2 (and the well-known fact that an irreducible unitary representation of $G(\A)$ is admissible, cf.\ \cite{flath}, Thm.\ 4.(2)) $ \calH_i^{\infty_\A}=\lim_{n}\left(\calH_i^{\infty_\R}\right)^{K_n} $ is a Casselman-Wallach representation of $ G(\A)$. Obviously, $\theta_i$ is an isomorphism between the $ (\g_\infty,K_\infty,G(\A_f)) $-modules $\calH_{i,(K_\infty)}^{\infty_\A}$ and $H_{i,(K_\infty)}^{\infty_\A}$. Thus, as argued in the proof of Lem.\ \ref{lem:cass_wall_reps}.\eqref{lem:cass_wall_reps:2} $\theta_i$ must extend to an equivalence of $ G(\A) $-representations (and hence, in particular, a bi-continuous map)
	\[ \bar\theta_i:\calH_i^{\infty_\A}\ira\Cl_{\calA_\J^\infty([G])}(H_{i,(K_\infty)}^{\infty_\A}). \]
	In order to complete the proof, it remains to show that $\bar\theta_i([f])=\theta_i([f])$ for every $ [f]\in\calH_i^{\infty_\A} $. If $ [f]\in\calH_i^{\infty_\A} $, then $ [f]\in\left(\calH_i^{\infty_\R}\right)^{K_n} $ for some $ n\in\N $, hence by \cite[Lem.\ 4]{hch66} there exists a sequence $ ([f_m])_{m\in\N}$, $f_m\in\left(\calH_i^{\infty_\R}\right)^{K_n}_{(K_\infty)} $ converging to $ [f] $ in $ \left(\calH_i^{\infty_\R}\right)^{K_n} $, hence in $ \calH_i^{\infty_\A} $, hence in $ \calH_i\subseteq L^2([G]) $ and hence, replacing the original sequence $ ([f_m])_{m\in\N} $ by a suitable subsequence and choosing (any) representatives for our classes in sight, the sequence $(f_m)_{m\in\N}$ converges almost everywhere on $ G(\A)$ to $f$. On the other hand, by the continuity of $\bar\theta_i $, $ \theta_i([f_m])=\bar\theta_i([f_m])\to \bar\theta_i([f]) $ in $ \calA_\J^\infty([G]) $, hence also pointwise everywhere on $ G(\A) $, cf.\ \S \ref{sect:LF-autom}. It follows that $ \bar\theta_i([f])=f$ almost everywhere on $ G(\A) $, which, together with the continuity of $ \bar\theta_i([f])$, cf.\ Rem.\ \ref{rem:phicont}, implies that $\bar\theta_i([f])=\theta_i([f])$, which proves the claim.
\end{proof}

\begin{rem}
Thm.\ \ref{thm:cusp} has the following consequence: Every cuspidal smooth-automorphic form $\varphi$ is a {\it finite} sum of smooth cuspidal functions $\varphi_i\in \mathcal H^{\infty_\A}_i$. For $K_\infty$-finite cuspidal automorphic forms this is well-known and an immediate consequence of \eqref{eq:117}, whereas it is in general wrong for elements of the space $L_{cusp,\J}^2([G])$. In view of the above mentioned, natural inclusions 
$$L_{cusp,\J}^2([G])\supsetneq \calA^\infty_{cusp,\J}([G])\supsetneq \calA_{cusp,\J}([G]) $$ 
this finiteness-statement hence amounts to the dictum that cuspidal smooth-automorphic forms $\varphi$ are more similar to $K_\infty$-finite cuspidal automorphic forms than to square-integrable cuspidal functions. Remarkably, this is so, though the quotient $\calA_\J^\infty([G])/\calA_\J([G])$ has uncountable dimension. %(as we assume the Axiom of Choice). 
\end{rem}

\subsection{The cuspidal support of a smooth-automorphic form}
We will now give a definition of the cuspidal support of a smooth-automorphic form. Our notion of cuspidal support -- which will be intrinsic to the smooth-automorphic setting -- will extend the usual definition for classical automorphic forms to the framework of smooth-automorphic forms. Thm.\ \ref{thm:cusp} will be a crucial ingredient in what follows. \\\\
Let $ \left\{P\right\} $ denote an associate class of parabolic $ F $-subgroups of $G$, represented by $P=L_PN_P\in\calP$. Given an ideal $\J$ of $\calZ(\g)$ of finite codimension, an {\it associate class} of cuspidal smooth-automorphic representations is represented by a pair $([\tilde\pi],\Lambda)$, where 
\begin{enumerate}
\item $[\tilde\pi]$ is an equivalence class of an irreducible cuspidal smooth-automorphic subrepresentation $\tilde\pi$ of $L_P(\A)$ and
\item $\Lambda:A^\R_P\rightarrow\C^*$ is a Lie group character, which is trivial on $A^\R_G$,
\end{enumerate}
such that the following compatibility-hypothesis is satisfied: Let $\lambda_0:=d\Lambda\in\check\a^G_{P,\C}$ be the derivative of $\Lambda$ and consider the irreducible smooth-automorphic subrepresentation $\pi:=e^{\langle\lambda_0, H_P(\cdot)\rangle} \cdot \tilde\pi$ of $L_P(\A)$. We suppose that the Weyl group orbit of the infinitesimal character of $\pi_\infty:=\bigcprojtp_{\vv \in S_\infty} \pi_\vv $ (cf.\ Thm.\ \ref{thm:TPthm}) annihilates $\J$. Here, $\J$ is viewed by means of the Harish-Chandra isomorphism as an ideal of the algebra $S(\check\a_{0,\C})^{W_G}$ of $W_G$-fixed elements in the symmetric algebra $S(\check\a_{0,\C})$ of $\check\a_{0,\C}$, cf.\ \S\ref{sect:liegrps}.\\\\ In fact, given a pair $([\tilde\pi],\Lambda)$ as above, the actual associate class of smooth-automorphic subrepresentations, represented by $([\tilde\pi],\Lambda)$ (or, equivalently, by $[\pi]$), is given by the collection $\varphi([\pi])=\{\varphi_Q([\pi])\}_{Q\in\{P\}}$ of finite sets of equivalence classes $\varphi_Q([\pi]):=\{[w\cdot\pi] \ | \ w\in W(L_P) \ \textrm{such that} \ wL_Pw^{-1} = L_Q \}$ of smooth-automorphic subrepresentations $w\cdot\pi$ of $L_Q(\A)$, where as usual $(w\cdot \pi)(\ell):=\pi(w^{-1}\ell w)$ for $\ell\in L_Q(\A)$.\\\\ Given $\J$, we denote by $\Phi_{\J,\left\{P\right\}}$ the set of all associate classes $\varphi([\pi])$, represented by a pair $([\tilde\pi],\Lambda)$ as above. \\\\ 
We point out that our notion of associate classes of cuspidal smooth-automorphic subrepresentations extends the usual notion of associate classes (cf.\ \cite{franke_schwermer}, \S 1.2) into the context of smooth-automorphic forms, i.e., the collections $\varphi([\pi_{(K_\infty)}])$ of finite sets of equivalence classes of $ (\g_\infty,K_\infty,G(\A_f)) $-modules of $K_\infty$-finite vectors in $\pi$ and its $W(L_P)$-conjugates coincide with the associate classes of cuspidal automorphic representations as defined in \cite{franke_schwermer}, \S 1.2 (for ideals $\J$ stemming from coefficients in automorphic cohomology; however, see also \cite{franke_schwermer}, Rem.\ 3.4): The verification of this claim relies crucially on the characterization of the irreducible cuspidal smooth-automorphic subrepresentations as the spaces of globally smooth vectors $\calH^{\infty_\A}_i$, $i\in I$, provided by Thm.\ \ref{thm:cusp} as applied to the Levi-subgroups $L_Q(\A)$, and the following lemma:

\begin{lem}\label{lem:egal}
For any two irreducible direct summands $\calH$ and $\calH'$ in the decomposition $L^2_{cusp}([L])=\widehat\bigoplus_{i\in I} \ \mathcal H_i,$ the following assertions are equivalent: 
\begin{enumerate}
\item $\calH\cong\calH'$
\item $\calH^{\infty_\A}\cong\calH'^{\infty_\A}$
\item $\calH^{\infty_\A}_{(K_\infty)}\cong\calH'^{\infty_\A}_{(K_\infty)}$
\end{enumerate}
\end{lem}
\begin{proof}
Recalling from Thm.\ \ref{thm:cusp} and Prop.\ \ref{prop:smmf} that $\calH^{\infty_\A}$ and $\calH'^{\infty_\A}$ are Casselman-Wallach representations of $G(\A)$, the equivalence of (2) and (3) is the assertion of Lem.\ \ref{lem:cass_wall_reps}.\eqref{lem:cass_wall_reps:2}. That (3) implies (1) follows from  well-known results on irreducible unitary representations of local groups $G(F_{\sf v})$, proved by Harish-Chandra, Bernstein and Godement, or, more precisely, from a combination of \cite{flath}, Thm.\ 3 \& 4, \cite{wallachI}, Thm.\ 3.4.11 and \cite{cartier}, Thm.\ 2.8 \& \S 2.8 {\it ibidem}.
\end{proof}

\noindent Having recalled this dictionary of definitions, let $\varphi([\pi])\in\Phi_{\J,\left\{P\right\}}$ and let $I^G_{P}(\widetilde\pi)$ be the space of all smooth, left $L(F)N(\A)A^\R_P$-invariant functions $f: G(\A)\ra\C$, such that for every $g\in G(\A)$ the function $l\mapsto f(lg)$ on $L(\A)$ is contained in 
$$\bigtoplus_{\substack{i\in I\\ \calH^{\infty_\A}_i\cong\tilde\pi}} \calH^{\infty_\A}_i\cong \tilde\pi^{m(\tilde\pi)},$$ 
where $m(\tilde\pi)$ denotes the finite multiplicity of $\tilde\pi$ in $\calA^\infty_{cusp,\J_L}([L])$ (in one -- and hence any -- $\calA^\infty_{cusp,\J_L}([L])$, into which $\tilde\pi$ embeds). By Lem.\ \ref{lem:egal}, this use of notation is consistent with our previous one, i.e.,
$$m(\tilde\pi)=\dim \Hom_{G(\A)}(\tilde\pi, \calA^\infty_{cusp,\J_L}([L]))=\dim \Hom_{G(\A)}( \calH_i, L_{cusp}^2([G])),$$
for each $\calH_i$, $i\in I$, such that $\calH^{\infty_\A}_i\cong\tilde\pi$. For a function $f\in I^G_{P}(\widetilde\pi)$, $\lambda\in\check\a^G_{P,\C}$ and $g\in G(\A)$ an {\it Eisenstein series} may be formally defined as

$$E_{P}(f,\lambda)(g):=\sum_{\gamma\in P(F)\backslash G(F)}
f(\gamma g)e^{\<\lambda+\rho_P,H_{P}(\gamma g)\>}.$$
If $f$ is $K_\infty$-finite, the so-defined Eisenstein series is known to converge absolutely and uniformly on compact subsets of $G(\A)\times \{\lambda\in\check\a^G_{P,\C}| \Re e(\lambda)\in\rho_P+\check\a^{G+}_P\}$. In this case, $E_{P}(f,\lambda)$ is an automorphic form. The map
$\lambda\mapsto E_{P}(f,\lambda)(g)$ can be analytically continued to a meromorphic function on all of $\check\a^G_{P,\C}$, cf. \cite{moewal}, Thm.\ IV.1.8 or \cite{langlandsLNM}, \S7. Its singularities (i.e., poles) lie along certain affine hyperplanes of the form $R_{\alpha, t}:=\{\xi\in\check\a^G_{P,\C}| (\xi,\alpha)=t\}$ for some constant $t$ and some root $\alpha\in\Delta(P,A_P)$, called ``root-hyperplanes'' (cf.\ \cite{moewal}, Prop. IV.1.11 (a) or \cite{langlandsLNM}, pp.\ 170--171). This entails the assertion that for each $\xi_0\in \check\a^G_{P,\C}$, there is a finite subset $\Delta_{\xi_0}\subseteq\Delta(P,A_P) $, such that the function 
$$q_{\xi_0}(\lambda):=\prod_{\alpha\in\Delta_{\xi_0}}\langle \lambda-{\xi_0},\check\alpha\rangle,$$
which is a non-zero, holomorphic function in $\lambda\in \check\a^G_{P,\C}$, has the property to make the assignment $ \check\a^G_{P,\C}\ra\C$ $\lambda \mapsto q_{\xi_0}(\lambda) E_P(f,\lambda)(g)$ holomorphic in a small neighbourhood of ${\xi_0}$ for all $K_\infty$-finite $f\in I_P^G(\tilde\pi)$ and $g\in G(\A)$.\\\\
Now, let ${S}(\check\a^G_{P,\C})$ be the symmetric algebra of $\check\a^G_{P,\C}$, viewed as the space of differential operators $\partial$ with constant coefficients on $\check\a^G_{P,\C} $, cf.\ \S\ref{sect:liegrps}. Then, at $\lambda_0=d\Lambda$ as above, the formal assignment defined by
$${\rm Eis_{\tilde\pi,\lambda_0}}(f\otimes \partial):=\partial(q_{\lambda_0}(\lambda) E_P(f,\lambda))|_{\lambda=\lambda_0}$$
turns out to be a well-defined map on the $K_\infty$-finite elements $f\in I^G_P(\tilde\pi)$ and $\partial\in{S}(\check\a^G_{P,\C})$ and we set for each $\varphi=\varphi([\pi])\in\Phi_{\J,\left\{P\right\}}$, 
$$\mathcal A_{\mathcal J,\{P\},\varphi}([G]):=Im({\rm Eis_{\tilde\pi,\lambda_0}}).$$
Its definition is independent of the choice of the representatives $P$ and $([\tilde\pi],\Lambda)$, thanks to the functional equations satisfied by the Eisenstein series considered, cf.\ \cite{moewal}, Thm.\ IV.1.10. Moreover, $\mathcal A_{\mathcal J,\{P\},\varphi}([G])$ is a $ (\g_\infty,K_\infty,G(\A_f)) $-submodule of $\mathcal A_{\mathcal J,\{P\}}([G])$, cf.\ \cite{franke_schwermer}, p.\ 771 and p.\ 773. Hence, combining Lem.\ \ref{lem:008}, Prop.\ \ref{prop:014} and Thm.\ \ref{thm:112}, 
$$\mathcal A^\infty_{\mathcal J,\{P\},\varphi}([G]):=\Cl_{\calA^\infty_{\J}([G])}(\mathcal A_{\mathcal J,\{P\},\varphi}([G]))$$
is a smooth-automorphic subrepresentation, lying inside the LF-compatible subrepresentation $\mathcal A^\infty_{\mathcal J,\{P\}}([G])$. The following is our last main result:

\begin{thm}\label{thm:cuspsup}
	For every $ P\in\calP $, we have the following $G(\A)$-equivariant decomposition into a locally convex direct sum of LF-compatible smooth-automorphic subrepresentations:
	\[ \calA^\infty_{\J,\left\{P\right\}}([G])=\bigtoplus_{\varphi\in\Phi_{\J,\left\{P\right\}}}\mathcal A^\infty_{\mathcal J,\{P\},\varphi}([G]). \]
\end{thm}
\begin{proof}
After our preparatory work above, we only need to observe that for every $ P\in\calP $, we have a decomposition 
\[ \calA_{\J,\left\{P\right\}}([G])=\bigoplus_{\varphi\in\Phi_{\left\{P\right\}}}\calA_{\J,\left\{P\right\},\varphi}([G]) \]
into a direct sum of $ (\g_\infty,K_\infty,G(\A_f)) $-submodules \cite{franke_schwermer}, Thm.\ 1.4 and Rem.\ 3.4, {\it loc.\ cit.} Thus, our Prop.\ \ref{prop:108} implies the desired result.
\end{proof}

\begin{rem}
As implied by Thm.\ \ref{thm:cuspsup}, taking the topological closure in $\calA^\infty_{\J}([G])$ of the image of what is given by forming all possible derivatives of residues of Eisenstein series, attached to the subspace of $K_\infty$-finite elements in the various induced representations $I^G_P(\tilde\pi)$, yields all smooth-automorphic forms. It will be interesting to know, whether this process can be inverted in the sense that the steps of taking topological closures and taking $K_\infty$-finite vectors can be interchanged, i.e., whether one may continuously extend the map ${\rm Eis_{\tilde\pi,\lambda_0}}$ to all of $I^G_P(\tilde\pi)$ and still obtain $\mathcal A^\infty_{\mathcal J,\{P\},\varphi}([G])$ as its image. Yet, to this end, one would need to specify a meaningful topology on $I^G_P(\tilde\pi)$ that reveals the space of $K_\infty$-finite vectors inside as a dense subspace. Interesting results in this direction are contained in \cite{erez_book}. However, in this reference a very much different topological approach, which is more suited to an ad-hoc analysis of induced representations rather than to the full spaces $\mathcal A^\infty_{\mathcal J,\{P\},\varphi}([G])$, was taken. We hope to report on this question in the forthcoming ``part 2'' of the present paper, which we will devote to the analytic continuation of smooth-automorphic Eisenstein series, i.e., to a smooth-automorphic version of the results in \cite{moewal}, IV. 
\end{rem}

\end{document}